\documentclass[11pt,final,oneside]{article}%
\usepackage{stix2}
\usepackage{appendix}
\usepackage{amsfonts}
\usepackage{amsmath}
\usepackage{amssymb}
\usepackage{amsxtra,amscd}
\usepackage[amsmath,hyperref,thmmarks]{ntheorem}
\usepackage{makeidx}
\usepackage{graphicx}
\usepackage{cprotect}
\usepackage[colorlinks=true,bookmarksnumbered=true,pdfpagemode=None,final]{hyperref}
\providecommand{\U}[1]{\protect\rule{.1in}{.1in}}
\theoremnumbering{arabic}
\theoremheaderfont{\scshape}
\RequirePackage{latexsym}
\theorembodyfont{\slshape}
\theoremseparator{.}
\newtheorem{X}{X}[section]

\newtheorem{conjecture}[X]{Conjecture}
\newtheorem{corollary}[X]{Corollary}

\newtheorem{lemma}[X]{Lemma}
\newtheorem{glemma}[X]{Goursat's Lemma}
\newtheorem{proposition}[X]{Proposition}

\newtheorem{theorem}[X]{Theorem}

\newtheorem*{mtconjecture}[X]{Mumford--Tate Conjecture}
\newtheorem*{tconjecture}[X]{Tate Conjecture}
\newtheorem*{fconjecture}[X]{Folklore Conjecture}
\newtheorem*{lconjecture}[X]{Lefschetz Standard Conjecture}
\newtheorem*{hconjecture}[X]{Hodge Standard Conjecture}
\theorembodyfont{\upshape}

\newtheorem{example}[X]{Example}

\newtheorem{plain}[X]{}

\newtheorem{remark}[X]{Remark}
\newtheorem{summary}[X]{Summary}
\theorembodyfont{\small}
\newtheorem{aside}[X]{Aside}

\theorembodyfont{\normalsize}
\theoremstyle{nonumberplain}
\theoremsymbol{\ensuremath{_\Box}}
\newtheorem{proof}{Proof}

\qedsymbol{\ensuremath{_\Box}}
\newcommand{\df}{\smash{\lower.12em\hbox{\textup{\tiny def}}}}

\usepackage{etex}

\usepackage{changepage}

\let\quoteOld\quote
\let\endquoteOld\endquote
\renewenvironment{quote}{\small\quoteOld}{\endquoteOld}

\usepackage[overload]{textcase}

\usepackage{xcolor}
\definecolor{bblue}{rgb}{0.0, 0.0, 0.6}

\usepackage{array}

\usepackage{microtype}

\usepackage{tikz}
\usepackage{tikz-cd}
\usetikzlibrary{matrix,arrows,positioning,decorations.pathmorphing}
\usetikzlibrary{decorations.markings,shapes.geometric}
\tikzset{commutative diagrams/column sep/Huge/.initial=12ex}

\usepackage{wrapfig}

\usepackage[shortlabels]{enumitem}
\setitemize[1]{label=$\diamond$\hspace{0.07in}}
\setlist{topsep=0.1em,itemsep=0.1em,parsep=0.1em}
\setdescription{font=\normalfont}

\usepackage{url}
\usepackage{verbatim}

\usepackage[notcite,notref]{showkeys}

\usepackage{mathtools}

\newcommand{\eb}[1]{{\itshape\bfseries#1}}
\renewcommand{\emph}{\eb}

\usepackage{titletoc,titlesec}
\contentsmargin{2.0em}
\dottedcontents{section}[2.3em]{}{1.8em}{1pc}
\usepackage[nottoc]{tocbibind}

\titleformat*{\section}{\LARGE\bfseries}
\titleformat*{\subsection}{\Large\itshape}
\titleformat*{\subsubsection}{\scshape}%\large
\titleformat*{\paragraph}{\itshape}

\usepackage{natbib}
\let\cite\citealt
\hypersetup{linkcolor=bblue,anchorcolor=bblue,citecolor=bblue,filecolor=bblue,urlcolor=bblue}
\hypersetup{pdftitle={The Conjectures of Grothendieck, Hodge, and Tate for Abelian varieties},pdfauthor={J.S. Milne}}
\hypersetup{pdfsubject={none},pdfkeywords={none}}

\usepackage[textwidth=5.5in,textheight=9.0in,centering,a4paper]{geometry}

\newcommand\babstract{\begin{abstract}}\newcommand\eabstract{\end{abstract}}
\newcommand{\bcomment}{\begin{comment}}\newcommand\ecomment{\end{comment}}
\newcommand{\bfootnotesize}{\begin{footnotesize}}\newcommand\efootnotesize{\end{footnotesize}}
\newcommand{\bquote}{\begin{quote}}\newcommand\equote{\end{quote}}
\newcommand{\bsmall}{\begin{small}}\newcommand\esmall{\end{small}}
\newcommand{\btable}{\begin{table}}\newcommand{\etable}{\end{table}}

\newcommand{\edocument}{\end{document}}

\newcommand{\tableoc}{\tableofcontents}%* removes "contents" from TOC

\renewcommand{\Gamma}{\varGamma}
\renewcommand{\Pi}{\varPi}
\renewcommand{\Sigma}{\varSigma}
\def\1{{1\mkern-7mu1}}

\DeclareMathOperator{\Aut}{Aut}

\DeclareMathOperator{\End}{End}

\DeclareMathOperator{\Gal}{Gal}
\DeclareMathOperator{\GL}{GL}
\DeclareMathOperator{\Hg}{Hg}
\DeclareMathOperator{\Hom}{Hom}
\DeclareMathOperator{\id}{id}

\DeclareMathOperator{\Ker}{Ker}

\DeclareMathOperator{\MT}{MT}
\DeclareMathOperator{\Nm}{Nm}

\DeclareMathOperator{\SU}{SU}

\DeclareMathOperator{\Tgt}{Tgt}

\DeclareMathOperator{\Tr}{Tr}

\newcommand{\Mot}{\mathsf{Mot}}%motives

\tikzcdset{arrow style=math font}
\begin{document}

\title{The Tate and Standard Conjectures for Certain Abelian Varieties}
\author{J.S.\ Milne}
\date{February 6, 2022}
\maketitle

\begin{abstract}
In two earlier articles, we proved that, if the Hodge conjecture is true for
\textit{all} CM abelian varieties over $\mathbb{C}$, then both the Tate
conjecture and the standard conjectures are true for abelian varieties over
finite fields. Here we rework the proofs so that they apply to a single
abelian variety. As a consequence, we prove (unconditionally) that the Tate
and standard conjectures are true for many abelian varieties over finite
fields, including abelian varieties for which the algebra of Tate classes is
\textit{not} generated by divisor classes. The article is partly expository.

\end{abstract}

\tableoc

\bigskip In earlier articles (1999b, 2002), we proved that, if the Hodge
conjecture holds for \textit{all} CM abelian varieties, then the Tate and
standard conjectures hold for abelian varieties over finite
fields.\footnote{Thus, if the Hodge standard conjecture fails for a single
abelian variety, then everything --- the Hodge, Tate, and Grothendieck
conjectures --- fails. From a more optimistic perspective, with the proof of
the algebraicity of Weil's classes in the first interesting case, namely, for
fourfolds with determinant $1$, we may hope that this will be proved for a
widening collection of abelian varieties. Then the methods of this paper will
show that the Tate and standard conjectures are also true for a widening
collection of abelian varieties. At some point we may dare to believe that the
conjectures are true for all abelian varieties.} In this article, we rework
the proofs so that they apply to a single abelian variety. We prove (Theorem
\ref{j11}) that if an abelian variety $A$ satisfies a certain condition
(\ref{blah}), then the Hodge conjecture for the powers of $A$ implies both the
Tate and standard conjectures for the powers of $A\text{ mod }p$. Using this,
we obtain a number of new results, among which is the theorem below.

We say that an abelian variety $A$ over an algebraically closed field $k$ of
characteristic zero (resp.~over $\mathbb{F}$) is neat if no power of it
supports an exotic Hodge class (resp.~an exotic Tate class). The theorems of
Lefschetz and Tate show that the Hodge and Tate conjectures hold for neat
abelian varieties. Let $\mathbb{Q}{}^{\mathrm{al}}$ denote the algebraic
closure of $\mathbb{Q}{}$ in $\mathbb{C}{}$ and $\mathbb{F}{}$ its residue
field at a prime dividing $p$.

\begin{theorem}
\label{j0}Let $A$ be a simple abelian variety over $\mathbb{Q}^{\mathrm{al}}$
of dimension $n$ with good reduction to a simple abelian variety $A_{0}$ over
$\mathbb{F}{}$, and let $B$ be a CM elliptic curve over $\mathbb{Q}%
{}^{\mathrm{al}}$ with good reduction $B_{0}$. Suppose that both $A$ and
$A_{0}$ are neat, but that neither $A\times B$ nor $A_{0}\times B_{0}$ is
neat.\footnote{To make the statement interesting.}

\begin{enumerate}
\item For some $m$, there is an imaginary quadratic field $Q\subset
\End^{0}(A\times B^{m})$ such that $(A\times B^{m},Q)$ is of Weil type.

\item If the Weil classes on $(A\times B^{m},Q)$ are algebraic, then, for all
$r,s\in\mathbb{N}$,

\begin{enumerate}
\item the Hodge conjecture holds for the abelian varieties $A_{\mathbb{C}{}%
}^{r}\times B_{\mathbb{C}{}}^{s}{}$;

\item the Tate conjecture holds for the abelian varieties $A^{r}\times B^{s}$;

\item the Tate conjecture holds for the abelian varieties $A_{0}^{r}\times
B_{0}^{s}$;

\item the standard conjectures hold for the abelian varieties $A_{0}^{r}\times
B_{0}^{s}{}$.
\end{enumerate}
\end{enumerate}
\end{theorem}

When $n=3$, a recent theorem of Markman shows that the Weil classes on
$(A\times B,Q)$ are algebraic. Thus we obtain new cases of the Hodge, Tate,
and standard conjectures. Apart from \cite{ancona2021}, these are the first
unconditional results on the Hodge standard conjecture since it was stated
over fifty years ago.

We give examples of pairs $A,B$ satisfying the hypotheses of the theorem.

We assume that the reader is familiar with \cite{milne2001}, especially
Appendix A, whose notation we adopt. In particular, $\iota$ or $a\mapsto
\bar{a}$ denotes complex conjugation on $\mathbb{C}$ and its subfields, and
$\mathbb{N}=\{0,1,\ldots\}$. For a CM algebra $E$, we let $U_{E}$ denote the
torus such that $U_{E}(\mathbb{Q}{})=\{a\in E^{\times}\mid a\bar{a}=1\}$. We
regard abelian varieties as objects of the category whose morphisms are
$\Hom^{0}(A,B)\overset{\df}{=}\Hom(A,B)\otimes\mathbb{Q}$. Throughout,
algebraic varieties are connected.

We assume that the reader is familiar with \cite{deligneM1982}, especially
\S 5, and \cite{deligne1989}, \S 5. The fundamental group $\pi(\mathsf{C})$ of
a tannakian category $\mathsf{C}$ is a group \textit{in} $\mathsf{C}$ such
that $\omega(\pi(\mathsf{C}))=\underline{\Aut}^{\otimes}(\omega)$ for all
fibre functors $\omega$, and $\gamma^{C}$ is the functor $\Hom(\1,-)$ from
$\mathsf{C}^{\pi(\mathsf{C)}}$ to vector spaces (an equivalence of tensor categories).

\section{Characteristic zero}

Let $A$ be an abelian variety over an algebraically closed field of
characteristic zero. By a Hodge class on $A$, we mean an absolute Hodge class
in the sense of \cite{deligne1982}. Such a class is said to be \emph{exotic}
if it is not in the $\mathbb{Q}$-algebra generated by the Hodge classes of
degree $1$. According to a theorem of Lefschetz, the nonexotic Hodge classes
are algebraic --- we call them \emph{Lefschetz classes}. We let $B^{\ast}(A)$
denote the $\mathbb{Q}$-algebra of Hodge classes on $A$ and $D^{\ast}(A)$ the
$\mathbb{Q}$-subalgebra generated by the divisor classes (the algebra of
Lefschetz classes).

\subsection{The Lefschetz group}

Let $A$ be an abelian variety over an algebraically closed subfield $k$ of
$\mathbb{C}$. The centralizer $C(A)$ of $\End^{0}(A)$ in $\End(H_{1}%
(A_{\mathbb{C}{}},\mathbb{Q}{}))$ is a $\mathbb{Q}{}$-algebra stable under the
involution $\dagger$ defined by an ample divisor $D$ of $A$, and the
restriction of $\dagger$ to $C(A)$ is independent of the choice $D$. The
\emph{Lefschetz group }$L(A)$ of $A$ is the algebraic group over $\mathbb{Q}%
{}$ such that
\[
L(A)(\mathbb{Q}{})=\big\{\alpha\in C(A)^{\times}\mid\alpha\alpha^{\dagger}%
\in\mathbb{Q}{}^{\times}\big\}.
\]
If $A$ is CM, i.e., if the $\mathbb{Q}{}$-algebra $\End^{0}(A)$ contains an
\'{e}tale $\mathbb{Q}{}$-subalgebra of degree $2\dim A$, then $C(A)$ is the
centre of $\End^{0}(A)$, and this definition makes sense over any
algebraically closed field (not necessarily of characteristic zero).

\subsection{The Mumford--Tate group}

Let $A$ be an abelian variety over an algebraically closed subfield $k$ of
$\mathbb{C}{}$, and let $V=H_{1}(A_{\mathbb{C}{}},\mathbb{Q}{})$. When we let
$L(A)$ act on the cohomology groups $H^{2r}(A_{\mathbb{C}}^{s},\mathbb{Q}%
{})(r)$, $r,s\in\mathbb{N}$, throught the homomorphism%
\[
\alpha\mapsto(\alpha,\alpha\alpha^{\dagger})\colon L(A)\rightarrow
\GL_{V}\times\mathbb{G}_{m}\text{,}%
\]
it becomes the algebraic subgroup of $\GL_{V}\times\mathbb{G}_{m}$ fixing the
Lefschetz classes on all powers of $A$. We define
\[
\MT(A)\subset M(A)\subset L(A)
\]
to be the algebraic subgroups of $L(A)$ fixing, respectively, the Hodge
classes and the algebraic classes on the powers of $A$. Then the Lefschetz
classes are \textit{exactly} the cohomology classes fixed by $L(A)$, and
similarly for the other groups. For the \emph{Mumford--Tate group} $\MT(A)$,
this is easy to prove; for $M(A)$, it follows from the fact that the abelian
motives modulo numerical equivalence form a tannakian category
(\cite{lieberman1968}, \cite{jannsen1992});\footnote{See the Appendix.} for
the Lefschetz group $L(A)$, it is proved by case-by-case checking
(\cite{milne1999lc}).

The kernel of the natural homomorphism $\MT(A)\rightarrow\mathbb{G}_{m}$
(resp.~$L(A)\rightarrow\mathbb{G}_{m})$ is called the \emph{Hodge group}
(resp.~the \emph{special Lefschetz group}) and denoted by $\Hg(A)$
(resp.~$S(A)$). The abelian variety $A$ is isogenous to a product
$A_{1}^{s_{1}}\times\cdots\times A_{m}^{s_{m}}$ with each $A_{i}$ simple and
no two isogenous, and%
\begin{align*}
\Hg(A)=  &  \Hg(A_{1}\times\cdots\times A_{m})=\text{a subproduct of
}\Hg(A_{1})\times\cdots\times\Hg(A_{m}),\\
S(A)=  &  S(A_{1}\times\cdots\times A_{m})=S(A_{1})\times\cdots\times
S(A_{m}).
\end{align*}
The kernel of $M(A)\rightarrow\mathbb{G}_{m}$ is denoted by $M^{\prime}(A)$.

\begin{proposition}
\label{j65}The following conditions on $A$ are equivalent:

\begin{enumerate}
\item no power of $A$ supports an exotic Hodge class;

\item $\MT(A)=L(A)$;

\item $\Hg(A)=S(A)$.
\end{enumerate}
\end{proposition}

\begin{proof}
The equivalence of (a) and (b) follows from the above discussion, and the
equivalence of (b) and (c) follows from the five-lemma. See \cite{milne1999lc}%
, 4.8.
\end{proof}

\noindent An abelian variety satisfying the equivalent conditions of the
proposition is said to be \emph{neat}.\footnote{In other words, the neat
abelian varieties are those for which the Hodge conjecture is true for trivial
reasons.} For example, if $A$ is an abelian variety $A$ such that
$E\overset{\df}{=}\End^{0}(A)$ is a field, then $A$ is neat in each of the
following cases: $E$ is totally real and $\dim(A)/[E\colon\mathbb{Q}]$ is odd;
$E$ is CM and $\dim(A)$ is prime; $E$ is imaginary quadratic and the
representation of $E\otimes\mathbb{C}{}$ on $\Tgt_{0}(A)$ is of the form
$m\rho\oplus n\bar{\rho}$ with $\gcd(m,n)=1$ (Tankeev; \cite{ribet1983}). It
follows that all simple abelian varieties of odd prime dimension are neat.
Abelian surfaces and products of elliptic curves are also neat.

\subsection{Abelian varieties of Weil type}

In the early 1960s, Mumford and Tate tried to prove the Hodge conjecture for
abelian varieties by showing that they are all neat, but then Mumford found a
nonneat simple abelian fourfold (\cite{pohlmann1968} \S 2). Mumford's variety
was CM, but as Weil (1977) explained, the important fact was that the abelian
variety was acted on by an imaginary quadratic field.

Let $A$ be an abelian variety over $\mathbb{C}$, and let $Q$ be an imaginary
quadratic subfield\footnote{By this we mean that it is a $\mathbb{Q}%
$-subalgebra; in particular, the identity elements coincide.} of the
$\mathbb{Q}{}$-algebra $\End^{0}(A)$. Let $\beta\in Q$ be such that
$\bar{\beta}=-\beta$, so $Q=\mathbb{Q}{}[\beta]$ and $\beta\bar{\beta}%
=b\in\mathbb{Q}{}$.

\begin{lemma}
\label{j75}There exists a polarization of $(A,Q)$, i.e., a polarization of $A$
whose Rosati involution stabilizes $Q$ and acts on it as complex conjugation.
\end{lemma}

\begin{proof}
Let $\psi$ be a Riemann form for $A$, i.e., an element $\psi$ of
$\Hom(\bigwedge_{\mathbb{Q}}^{2}V,\mathbb{Q}{})\simeq H^{2}(A,\mathbb{Q})$
such that $\psi_{\mathbb{R}{}}(Jx,Jy)=\psi_{\mathbb{R}}(x,y)$ and
$\psi_{\mathbb{R}}(x,Jx)>0$ for $x,y\in V_{\mathbb{R}}$, $x>0$. Then $\beta$
acts on $B^{1}(A)(-1)\subset H^{2}(A,\mathbb{Q})$ with eigenvalues in
$\{d,-d\}$. Let $\psi=\psi_{+}+\psi_{-}$ be the decomposition of $\psi$ into
eigenvectors. Then $\psi_{+}$ is again a Riemann form, and the condition
$\beta^{\ast}\psi_{+}=d\psi_{+}$ implies that $\psi_{+}(\beta x,y)=\psi
_{+}(x,\bar{\beta}y)$, as required.
\end{proof}

Let $\lambda$ be a polarization of $(A,Q)$, and
\[
\psi\colon H_{1}(A,\mathbb{Q}{})\times H_{1}(A,\mathbb{Q}{})\rightarrow
\mathbb{Q}{}%
\]
its Riemann form. Define $\phi\colon H_{1}(A,\mathbb{Q})\times H_{1}%
(A,\mathbb{Q})\rightarrow Q$ by%
\[
\phi(x,y)=\psi(x,\beta y)+\beta\psi(x,y).
\]
Then $\phi$ is the unique hermitian form on the $Q$-vector space
$H_{1}(A,\mathbb{Q})$ such that
\[
\phi(x,y)-\overline{\phi(x,y)}=2\beta\psi(x,y).
\]
The discriminant of $\phi$ is an element of $\mathbb{Q}^{\times}%
/\Nm(Q^{\times})$, called the \emph{determinant} of $(A,Q,\lambda)$.

\begin{proposition}
\label{j60} Let $A$ be an abelian variety of dimension $2m$ and $Q$ an
imaginary quadratic subfield of $\End^{\circ}(A)$. The following conditions on
$(A,Q)$ are equivalent:

\begin{enumerate}
\item $\mathrm{Tgt}_{0}(A)$ is a free $Q\otimes_{\mathbb{Q}}\mathbb{C}$-module;

\item the one-dimensional $Q$-vector space\footnote{As before, let
$V=H_{1}(A,\mathbb{Q})$. We can identify $H^{2m}(A,\mathbb{Q})$ with the space
of $\mathbb{Q}$-multilinear alternating forms $V\times\cdots\times
V\rightarrow\mathbb{Q}$ and $W(A,Q)(-m)$ with the subspace of those that are
$Q$-balanced.}
\[
W(A,Q)\overset{\df}{=}\big(\bigwedge\nolimits_{Q}^{2m}H^{1}(A,\mathbb{Q}%
)\big)(m)\subset H^{2m}(A,\mathbb{Q})(m)
\]
consists of Hodge classes.
\end{enumerate}
\end{proposition}

\begin{proof}
Straightforward; see Deligne 1982, 4.4.
\end{proof}

\noindent A pair $(A,Q)$ satisfying the equivalent conditions of the
proposition is said to be of \emph{Weil type}. The elements of $W(A,Q)$%
\textrm{ are }the \emph{Weil classes} on $A$. The hermitian form $\phi$
attached to a polarization $\lambda$ of $(A,Q)$ has signature $(m,m)$, and so
\[
(-1)^{m}\cdot\det(A,Q,\lambda)>0.
\]

\begin{proposition}
For any imaginary quadratic field $Q$, integer $m\geq1$, and element
$a\in\mathbb{Q}^{\times}/\Nm(Q^{\times})$ with sign $(-1)^{m}$, the
$2m$-dimensional polarized abelian varieties $(A,Q,\lambda)$ of Weil type with
determinant $a$ form an $m^{2}$-dimensional family.
\end{proposition}

\begin{proof}
The underlying variety of the family is the connected Shimura variety attached
to $\SU(\phi)$. See \cite{weil1977}, \cite{deligne1982}, 4.8, or
\cite{geemen1994}.
\end{proof}

Weil showed that, in general, the Weil classes are exotic, and suggested that
they formed a good test case for the Hodge conjecture. The first interesting
case is $\dim(A)=4$. Concerning this, there is the following result.

\begin{theorem}
[Markman]\label{j14}Let $(A,Q,\lambda)$ be a polarized abelian fourfold of
Weil type with determinant $1$ in $\mathbb{Q}^{\times}/\Nm(Q^{\times})$. Then
the Weil classes on $A$ are algebraic.
\end{theorem}

\begin{proof}
See \cite{markman2021}, 1.5. The proof uses methods from the deformation
theory of hyperk\"ahler manifolds.
\end{proof}

\begin{proposition}
\label{j7}Let $(A\times B,Q)$ be of Weil type, where $A$ is an abelian
$n$-fold and $B$ an elliptic curve. Then there exists a polarization $\lambda$
of $(A\times B,Q)$ with determinant $(-1)^{(n+1)/2}$ modulo $\Nm(Q^{\times})$.
\end{proposition}

\begin{proof}
Let $\lambda_{A}$ (resp.~$\lambda_{B}$) be a polarization of $(A,Q)$
(resp.~$(B,Q)$). Let $\phi_{A}$ and $\phi_{B}$ be the corresponding $Q$-valued
hermitian forms. Then $\lambda_{A}\times\lambda_{B}$ is a polarization of
$(A\times B,Q)$ with hermitian form $\phi_{A}\oplus\phi_{B}$, which has
determinant some $a\in\mathbb{Q}{}^{\times}/\Nm(Q^{\times})$ with
$(-1)^{(n+1)/2}a>0$. There exists a $c\in\mathbb{Z}$, $c>0$, such that
$ac=(-1)^{(n+1)/2}$ in $\mathbb{Q}^{\times}/\Nm(Q^{\times})$. Now $\lambda
_{A}\times c\lambda_{B}$ is a polarization of $(A\times B,Q)$ with determinant
$(-1)^{(n+1)/2}$ modulo $\Nm(Q^{\times})$.
\end{proof}

\begin{aside}
\label{j76}In the above discussion, $Q$ can be replaced by any CM field. If
all Weil classes in this more general sense are algebraic, then the Hodge
conjecture holds for all CM abelian varieties (\cite{deligne1982}, \S 5).
\end{aside}

\subsection{Almost-neat abelian varieties}

Let $(A,Q)$ be of Weil type. Then $S(A)$ acts on $W(A,Q)$ through a
\textquotedblleft determinant\textquotedblright\ homomorphism $\rho\colon
S(A)\rightarrow U_{Q}$, and $\Hg(A)$ is contained in the kernel of $\rho$. We
say that the pair $(A,Q)$ is \emph{almost-neat} if the sequence%
\begin{equation}
1\rightarrow\Hg(A)\longrightarrow S(A)\overset{\rho}{\longrightarrow}%
U_{Q}\rightarrow1 \label{j66}%
\end{equation}
is exact. As $W(A,Q)$ has weight $0$ and $L(A)=w(\mathbb{G}_{m})\cdot S(A)$,
it also acts on $W(A,Q)$ through a homomorphism $\rho\colon L(A)\rightarrow
U_{Q}$. A pair $(A,Q)$ of Weil type is almost-neat if and only if the sequence%
\begin{equation}
1\rightarrow\MT(A)\rightarrow L(A)\overset{\rho}{\longrightarrow}%
U_{Q}\rightarrow1 \label{j67}%
\end{equation}
is exact. If $(A,Q)$ is almost-neat, then the Weil classes on $A$ are exotic.

\begin{theorem}
\label{j63}If $(A,Q)$ is almost-neat and the Weil classes on $A$ are
algebraic, then the Hodge conjecture holds for $A$ and its powers.
\end{theorem}

\begin{proof}
As the Weil classes on $A$ are algebraic, $M(A)$ is contained in the kernel of
$\rho$, and so equals $\MT(A)$. Thus the Hodge classes on the powers of $A$
coincide with the algebraic classes.
\end{proof}

Recall that that $W(A,Q)$ has dimension $1$ as a $Q$-vector space. If $W(A,Q)$
contains a single nonzero algebraic class, then it consists of algebraic
classes because the action of the endomorphisms of $A$ on its cohomology
groups preserves algebraic classes.

\begin{lemma}
\label{j58}Let $E$ be a CM field and $Q$ an imaginary quadratic subfield of
$E$. Then%
\[
U_{E/Q}\overset{\df}{=}\Ker(\Nm_{E/Q}\colon U_{E}\rightarrow U_{Q})
\]
is a subtorus of $U_{E}$ of codimension $1$, and every subtorus of $U_{E}$ of
codimension $1$ is of this form for a unique imaginary quadratic subfield of
$E$.
\end{lemma}

\begin{proof}
That $U_{E/Q}$ is a subtorus of codimension $1$ can be checked on the
character groups. Conversely, if $U$ is a one-dimensional quotient of $U_{E}$,
then its splitting field $Q$ is an imaginary quadratic extension of
$\mathbb{Q}{}$, and $U=U_{Q}$. That the kernel of $U_{E}\rightarrow U$ is
isomorphic to $U_{E/Q}$ can be checked on the character groups.\footnote{This
lemma is well known.}
\end{proof}

\begin{proposition}
\label{j72}Let $A$ be a neat simple abelian variety and $B$ a CM elliptic
curve such that $A\times B$ is not neat.

(a) There is an exact sequence%
\[
1\rightarrow\Hg(A\times B)\rightarrow S(A\times B)\rightarrow U\rightarrow1
\]
in which the projection $S(A)\rightarrow U$ is surjective and $S(B)\rightarrow
U$ an isomorphism.

(b) There exists an embedding of $Q\overset{\df}{=}\End^{0}(B)$ into
$\End^{0}(A)$ and an $m\in\mathbb{N}{}$ such that $(A\times B^{m},Q)$ is of
Weil type (hence almost-neat${}$).
\end{proposition}

\begin{proof}
Because $A$ is neat but $A\times B$ is not,%
\[
S(A)\times S(B)=S(A\times B)\supsetneqq\Hg(A\times B)\twoheadrightarrow
\Hg(A)=S(A).
\]
It follows that $S(A)\times S(B)=\Hg(A\times B)\cdot S(B)$, and so there is an
exact sequence%
\[
1\rightarrow\Hg(A\times B)\rightarrow S(A)\times S(B)\rightarrow
U\rightarrow1
\]
in which $U$ one-dimensional and the projection $S(B)\rightarrow U$ is
surjective. The kernel of $S(B)\rightarrow U$ is trivial because $\Hg(A\times
B)$ is connected.

The projection $S(A)\rightarrow U$ is surjective, because otherwise
$\Hg(A\times B)=S(A)^{\circ}$, which is not possible because $\Hg(A\times
B)\rightarrow\Hg(B)=S(B)$ is surjective.

It follows that $A$ is of type IV, because otherwise the algebraic group
$S(A)^{\circ}$ would be semisimple (\cite{milne1999lc}, \S 2), and so the
centre of $\End^{0}(A)$ is a CM-field $E$. The centre of $C(A)$ equals that of
$\End^{0}(A)$,\footnote{Let $A$ and $B$ be subalgebras of an algebra, each the
centralizer of the other. Then $A\cap B$ is the common centre of $A$ and $B$.}
and the homomorphism $S(A)\rightarrow U$ induces a surjection $U_{E}%
\rightarrow U$. According to the Lemma \ref{j58}, there is an imaginary
quadratic field $Q$ and an embedding of $Q$ into $E$ such that $U_{E}%
\rightarrow U$ is given by the norm map $E\rightarrow Q$. Clearly
$Q\approx\End^{0}(B)$.

Let $\rho\colon Q\rightarrow\mathbb{C}$ be an embedding, and let the
representation of $Q\otimes\mathbb{C}$ on $\Tgt_{0}(A)$ be $n_{1}\rho\oplus
n_{2}\bar{\rho}$. We may suppose that $n_{1}>n_{2}$ (otherwise replace $\rho$
with $\bar{\rho}$), and let $m=n_{1}-n_{2}$. Choose the isomorphism
$Q\rightarrow\End^{0}(B)$ so that $Q\otimes\mathbb{C}$ acts on $\Tgt_{0}(B)$
as $\bar{\rho}$. When we let $Q$ act diagonally on $A\times B^{m}$, the pair
$(A\times B^{m},Q)$ is of Weil type, and (a) shows that it is almost-neat.
\end{proof}

Let $A$ be an abelian variety over $\mathbb{C}{}$ and $Q$ an imaginary
quadratic subfield of $\End^{0}(A)$. Let $n_{1}\rho\oplus n_{2}\bar{\rho}$ be
the representation of $Q\otimes\mathbb{C}{}$ on $\mathrm{Tgt}_{0}(A)$. If
$Q=\End^{0}(A)$, then $n_{1},n_{2}>0$ (Shimura; see \cite{ribet1983}, p.~525).
If $A$ is CM, then $n_{1},n_{2}>0$ unless $A$ is isogenous to a power of an
elliptic curve.

\begin{corollary}
\label{j74}Let $A$ be a simple abelian threefold and $B$ a CM elliptic curve.
The Hodge conjecture holds for all varieties $A^{r}\times B^{s}$,
$r,s\in\mathbb{N}$.
\end{corollary}

\begin{proof}
Recall that $A$ and $B$ are neat. If $A\times B$ is neat, then certainly the
Hodge conjecture holds for the varieties $A^{r}\times B^{s}$. Otherwise, there
exists an imaginary quadratic subfield $Q$ of $\End^{0}(A\times B)$ such that
$(A\times B,Q)$ is almost-neat (see \ref{j72}). The pair $(A\times B,Q)$
admits a polarization $\lambda$ such that $\det(A\times B,Q,\lambda)=1$ in
$\mathbb{Q}{}^{\times}/\Nm(Q^{\times})$ (see \ref{j7}), and so the Weil
classes on $A\times B$ are algebraic (\ref{j14}). Therefore the Hodge
conjecture holds for the varieties $A^{r}\times B^{s}$ (see \ref{j63}).
\end{proof}

Let $A$ and $B$ be as in the corollary, but with $\End^{0}(B)=\mathbb{Q}{}$.
If $A$ is not of type IV or is CM, then $A\times B$ is neat (\cite{hazama1989}%
, 0.1, 3.1). In the remaining case, $\End^{0}(A)$ is an imaginary quadratic
field, and it follows from Goursat's lemma that $\Hg(A\times B)=\Hg(A)\times
\Hg(B)$, and so $A\times B$ is again neat.

\begin{example}
\label{j77}Let $E$ be a CM field of degree $2m$, $m\geq3$, over $\mathbb{Q}{}$
containing an imaginary quadratic field $Q$. Choose an embedding $\rho
_{0}\colon Q\rightarrow\mathbb{Q}{}^{\mathrm{al}}$, and let $\{\varphi
_{0},\ldots,\varphi_{m-1}\}$ be the set of extensions of $\rho_{0}$ to $E$.
Then $\Phi_{0}\overset{\df}{=}\{\varphi_{0},\iota\circ\varphi_{1},\ldots
,\iota\circ\varphi_{m-1}\}$ is a CM-type on $E$, and we let $(A,E)$ denote an
abelian variety over $\mathbb{C}{}$ of CM-type $(E,\Phi_{0})$.

Let $(B,Q)$ be an elliptic curve over $\mathbb{C}{}$ of CM-type $(Q,\rho_{0}%
)$, and let $Q$ act diagonally on $A\times B^{m-2}$. Then%
\[
\Tgt_{0}(A\times B^{m-2})\simeq\Tgt_{0}(A)\oplus(m-2)\Tgt_{0}(B)
\]
is a free $Q\otimes\mathbb{C}{}$-module, and so $(A\times B^{m-2},Q)$ is of
Weil type.

The abelian variety $A$ is neat, and the pair $(A\times B^{m-2},Q)$ is
almost-neat (\cite{milne2001}). In particular, if $m=3$, then the Hodge
conjecture holds for the varieties $A^{r}\times B^{s}$, $r,s\in\mathbb{N}{}$.
\end{example}

\begin{remark}
\label{j78}Let $A$ be a neat abelian variety (not necessarily simple) and $B$
a CM elliptic curve such that $A\times B$ is not neat. Then there exists a
simple isogeny factor $A^{\prime}$ of $A$ such $A^{\prime}\times B$ is not
neat, and so Proposition \ref{j72} shows that there exists a homomorphism from
$Q\overset{\df}{=}\End^{0}(B)$ into a direct factor of the centre of
$\End^{0}(A)$.
\end{remark}

\subsection{Hodge classes on general abelian varieties of Weil type}

Let $(A,Q,\lambda)$ be a polarized abelian variety of Weil type, and let
$W(A,Q)$ be the $Q$-vector space of Weil classes on $A$. The action of $S(A)$
on $W(A,Q)$ defines a homomorphism $\rho\colon S(A)\rightarrow U_{Q}$. Let
$\phi\colon V\times V\rightarrow Q$ be the hermitian form attached to
$(A,Q,\lambda)$. There is an exact sequence of algebraic groups%
\begin{equation}
1\rightarrow\SU(\phi)\rightarrow\mathrm{U}(\phi)\overset{\det}{\longrightarrow
}U_{Q}\rightarrow1, \label{je1}%
\end{equation}
where $\mathrm{U}(\phi)$ is the subgroup of $\GL_{Q}(V)$ of elements fixing
$\phi$.

\begin{theorem}
[Weil]\label{j51}If $(A,Q,\lambda)$ is general,\footnote{That is, in the
complement of a countable union of proper closed subvarieties of the moduli
space.} then the sequence%
\[
1\rightarrow\Hg(A)\longrightarrow S(A)\overset{\rho}{\longrightarrow}%
U_{Q}\rightarrow1\text{,}%
\]
coincides with the exact sequence (\ref{je1}). In particular, $(A,Q)$ is
almost neat.
\end{theorem}

\begin{proof}
See \cite{weil1977}; also \cite{geemen1994}.
\end{proof}

In a family of abelian varieties, the Mumford-Tate group stays constant on the
complement of a countable union of closed subvarieties where it can only
shrink (see, for example, \cite{milne2013m}, \S 6). The theorem says that, in
the Weil family, the general Mumford--Tate and Lefschetz groups are as large
as possible given the obvious constraints.

\begin{remark}
\label{j50}If $(A,Q,\lambda)$ is general, then it follows from invariant
theory that%
\[
B^{\ast}(A)=D^{\ast}(A)\oplus W(A,Q).
\]
See, for example, \cite{geemen1994}, 6.12.
\end{remark}

\begin{theorem}
\label{j52}If $(A,Q,\lambda)$ is general, then the Weil classes are exotic,
and the Hodge conjecture holds for the powers of $A$ if they are algebraic.
\end{theorem}

\begin{proof}
This is obvious from the exact sequence%
\[
1\rightarrow\MT(A)\longrightarrow L(A)\overset{\rho}{\longrightarrow}%
U_{Q}\rightarrow1\text{.}%
\]

\end{proof}

\begin{corollary}
\label{j53}Let $(A,Q,\lambda)$ be a general polarized abelian fourfold of Weil
type. If $\det(A,Q,\lambda)=1$ in $\mathbb{Q}{}^{\times}/\Nm(Q^{\times})$,
then the Hodge conjecture holds for the powers of $A$.
\end{corollary}

\begin{proof}
According to Theorem \ref{j14}, the Weil classes on $A$ are algebraic.
\end{proof}

\begin{remark}
\label{j80}Let $(A_{i},Q_{i},\lambda_{i})$, $1\leq i\leq n$, be general
polarized abelian varieties of Weil type such that no two of the algebraic
groups $\SU(\phi_{i})$ are isomorphic. Then
\[
\Hg(A)=\Hg(A_{1})\times\cdots\times\Hg(A_{n})
\]
(Goursat's lemma). Hence, if the $(A_{i},Q_{i},\lambda_{i})$ are fourfolds
with determinant $1$, then the Hodge conjecture holds for all varieties of the
form $A_{1}^{m_{1}}\times\cdots\times A_{n}^{m_{n}}$, $m_{i}\in\mathbb{N}$.
\end{remark}

\begin{summary}
\label{j84}Let $A$ be a simple abelian fourfold. Either $A$ is neat, so the
Hodge conjecture holds for the powers of $A$, or there exists an imaginary
quadratic field $Q\subset\End^{0}(A)$ such that $(A,Q)$ is of Weil type. In
the second case, the Hodge conjecture holds for the powers of $A$ if there
exists a polarization $\lambda$ such that $\det(A,Q,\lambda)=1$ and
$(A,Q,\lambda)$ is general.
\end{summary}

\subsection{The Tate conjecture}

Let $X$ be a smooth projective variety over an algebraically closed field $k$
of of characteristic zero. Recall that $H^{2i}(X,\mathbb{Q}{}_{\ell}(i))\simeq
H^{2i}(\rho X,\mathbb{Q}_{\ell}(i))$ for all embeddings $\rho\colon
k\hookrightarrow\mathbb{C}{}$.

\begin{conjecture}
[Deligne]\label{j85}Let $\gamma\in H^{2i}(X,\mathbb{Q}{}_{\ell}(i))$ some
$\ell$. If $\gamma$ becomes a Hodge class\footnote{That is, an element of
$H^{2i}(\rho X,\mathbb{Q}(i))\cap H^{0,0}\subset H^{2i}(\rho X,\mathbb{Q}%
{}_{\ell}(i))$.} in $H^{2i}(\rho X,\mathbb{Q}{}_{\ell}(i))$ for one $\rho$,
then it does so for all $\rho$.
\end{conjecture}

\begin{theorem}
[Deligne]\label{j86}Conjecture \ref{j85} is true for abelian varieties.
\end{theorem}

\begin{proof}
This is the main theorem of \cite{deligne1982}
\end{proof}

From now on, the field $k$ will be the algebraic closure of a subfield
finitely generated over $\mathbb{Q}$. Let $X_{1}$ be a model of $X$ over a
finitely generated subfield $k_{1}$ of $k$ with algebraic closure $k$. An
element of the \'{e}tale cohomology group $H^{2i}(X,\mathbb{Q}{}_{\ell})(i)$
is a \emph{Tate class} if it is fixed by some open subgroup of $\Gal(k/k_{1}%
)$. This definition is independent of the choice of the model $X_{1}/k_{1}$. A
Tate class is \emph{exotic} if it is not in the $\mathbb{Q}{}_{\ell}$-algebra
generated by the Tate classes of degree $1$. According to a theorem of
Faltings (1983\nocite{faltings1983}), the nonexotic Tate classes on an abelian
variety are algebraic, i.e., in the $\mathbb{Q}{}_{\ell}$-span of the
cohomology classes of algebraic cycles.

\begin{tconjecture}
All $\ell$-adic Tate classes on $X$ are algebraic.
\end{tconjecture}

\begin{theorem}
[Piatetski-Shapiro, Deligne]\label{j87}Let $A$ be an abelian variety over
$k\subset\mathbb{C}$. If the Tate conjecture holds for $A$, then the Hodge
conjecture holds for $A_{\mathbb{C}}$.
\end{theorem}

\begin{proof}
Let $V\subset H^{2i}(A,\mathbb{Q}{}_{\ell}(i))$ be the $\mathbb{Q}{}$-subspace
spanned by the classes on $A$ that become Hodge on $A_{\mathbb{C}{}}$. Let
$A_{1}$ be a model of $A$ over a finitely generated subfield $k_{1}$ of $k$
with algebraic closure $k$. Theorem \ref{j86} implies that the action of
$\Gal(k/k_{0})$ on $H^{2i}(A,\mathbb{Q}{}_{\ell}(i))$ stabilizes $V$. As $V$
is countable, it follows that the action factors through a finite quotient.
Therefore $V$ consists of Tate classes, which are algebraic if the Tate
conjecture holds for $A$.
\end{proof}

\begin{remark}
\label{j88}Let $X$ be a smooth projective variety over $k$. We say that
$\gamma\in H^{2i}(X,\mathbb{Q}{}_{\ell}(i))$ is \emph{absolutely Hodge} if it
becomes Hodge under every embedding $\rho\colon k\hookrightarrow\mathbb{C}{}$.
The argument in the proof of \ref{j87} shows that, if the Tate conjecture
holds for $X$, then all absolutely Hodge classes on $X$ are algebraic.
Therefore,%
\[
\text{Deligne conjecture}+\text{Tate conjecture}\implies\text{Hodge
conjecture.}%
\]

\end{remark}

Let $A$ be an abelian variety over $k$, and let $A_{1}$ be a model of $A$ over
a finitely generated subfield $k_{1}$ with algebraic closure $k$. Some open
subgroup $U$ of $\Gal(k/k_{1})$ acts trivially on the Hodge classes in
$\bigoplus_{r,s}H^{2r}(A^{s},\mathbb{Q}_{l})(r)$ (see the proof of Theorem
\ref{j87}), and it follows that $\MT(A)(\mathbb{Q}_{\ell})$ contains $U$.

\begin{mtconjecture}
\label{j89} The algebraic group $\MT(A)$ is generated by the subgroup $U$,
i.e., if $G$ is an algebraic subgroup of $\MT(A)$ such that $G(\mathbb{Q}%
{}_{\ell})\supset U$, then $G=\MT(A)$.
\end{mtconjecture}

If the conjecture is true for one $U$ contained in $\MT(A)(\mathbb{Q}_{\ell}%
)$, then it is true for all $U$.

\begin{theorem}
\label{j92}If the Mumford--Tate conjecture is true for $A$, then the Tate
conjecture holds for $A$ if and only if the Hodge conjecture holds for
$A_{\mathbb{C}{}}$.
\end{theorem}

\begin{proof}
Since the Tate classes in $H^{2\ast}(A,\mathbb{Q}_{\ell}(\ast))$ are those
fixed by any sufficiently small $U$, and Hodge classes are those fixed by
$\MT(A)$, equivalently $\MT(A)(\mathbb{Q}_{\ell})$, this is obvious.
\end{proof}

The Mumford--Tate conjecture is known for many abelian varieties, for example,
for elliptic curves, abelian varieties of prime dimension (many authors;
\cite{chi1991}), most abelian fourfolds (\cite{lesin1994}), and all CM abelian
varieties (Shimura, Taniyama; \cite{pohlmann1968}). It is true for a product
of abelian varieties if it is true for each factor (\cite{vasiu2008},
\cite{commelin2019}).

\section{Characteristic $p$}

\subsection{Statement of the folklore conjecture}

Let $X$ be a smooth projective variety over an algebraically closed field $k$,
and let $\ell$ be a prime number distinct from the characteristic of $k$.

\begin{fconjecture}
\label{j32}Numerical equivalence coincides with $\ell$-adic homological
equivalence in the cohomology of $X$.\footnote{This was stated by Tate in the
talk at Woods Hole, 1964, in which he announced his conjectures, and so can be
considered to be part of the Tate conjectures. It is also a consequence of the
standard conjecture, and is sometimes referred to as the homological standard
conjecture.}
\end{fconjecture}

In characteristic zero, this has been proved for abelian varieties
(\cite{lieberman1968}). In characteristic $p$, we have only the following result.

\begin{theorem}
[Clozel]\label{j35}Let $A$ be an abelian variety over $\mathbb{F}{}$
(algebraic closure of the field of $p$ elements). There exists a set $s(A)$ of
primes $\ell$ of density $>0$ such that the folklore conjecture holds for $A$
and the $\ell$ in $s(A)$.
\end{theorem}

\begin{proof}
This is the main theorem of \cite{clozel1999}.
\end{proof}

\noindent The set $s(A)$ can be chosen to depend only on the set of simple
isogeny factors of $A$ (\cite{milne2001}, B.2). In particular, $s(A)=s(A^{n})$.

\subsection{Statement of the Tate conjecture over $\mathbb{F}{}$}

Let $\ell$ be a prime number $\neq p$. Let $X$ be a smooth projective variety
over $\mathbb{F}{}$, and let $X_{1}$ be a model of $X$ over a finite subfield
$\mathbb{F}{}_{q}$ of $\mathbb{F}{}$. An element of the \'{e}tale cohomology
group $H^{2i}(X,\mathbb{Q}{}_{\ell})(i)$ is a \emph{Tate class} if it is fixed
by an open subgroup of $\Gal(\mathbb{F}{}/\mathbb{F}{}_{q})$. This definition
is independent of the choice of the model $X_{1}/\mathbb{F}{}_{q}$. A Tate
class is \emph{exotic} if it is not in the $\mathbb{Q}{}_{\ell}$-algebra
generated by the Tate classes of degree $1$. According to a theorem of Tate
(1966\nocite{tate1966e}), the nonexotic Tate classes on an abelian variety are
algebraic, i.e., in the $\mathbb{Q}{}_{\ell}$-span of the cohomology classes
of algebraic cycles.

\begin{tconjecture}
\label{j31}All $\ell$-adic Tate classes on $X$ are algebraic.
\end{tconjecture}

\begin{theorem}
\label{j26} Let $X$ be a smooth projective variety over $\mathbb{F}$. If the
Tate and folklore conjectures are true for one $\ell\neq p$, then they are
true for all.
\end{theorem}

\begin{proof}
Folklore; see \cite{tate1994}, 2.9.
\end{proof}

\subsection{Tate classes on abelian varieties}

Let $A$ be an abelian variety over $\mathbb{\mathbb{F}{}}$. A model $A_{1}$ of
$A$ over a finite subfield $\mathbb{F}{}_{q}$ of $\mathbb{F}{}$ defines a
Frobenius element $\pi_{1}\in\End^{0}(A)$. The group $P(A)$ is defined to be
the smallest algebraic subgroup of $L(A)$ containing some power of $\pi_{1}$
--- it is independent of the choice of the model $A_{1}/\mathbb{F}{}_{q}$.
There is a canonical homomorphism $P(A)\rightarrow\mathbb{G}_{m}$, and we let
$P^{\prime}(A)$ denote its kernel.

\begin{theorem}
\label{j9}Let $A$ be an abelian variety over $\mathbb{F}{}$. The following
conditions on $A$ are equivalent:

\begin{enumerate}
\item no power of $A$ supports an exotic Tate class;

\item $P(A)=L(A)$.
\end{enumerate}
\end{theorem}

\begin{proof}
Let $\ell$ be a prime $\neq p$. Almost by definition, the Tate classes are the
$\ell$-adic cohomology classes fixed by $P(A)$. On the other hand, the
cohomology classes fixed by $L(A)$ are exactly those in the $\mathbb{Q}_{\ell
}$-algebra generated by the Tate classes of degree 1 (\cite{milne1999lc},
3.2). Therefore (b) implies (a), and the converse is true because both groups
are determined by their fixed tensors.
\end{proof}

\noindent An abelian variety over $\mathbb{F}$ satisfying the equivalent
conditions of the proposition is said to be \emph{neat}. Many abelian
varieties are known to be neat, for example, products of elliptic curves, and
abelian varieties satisfying certain conditions on their Newton polygons
(Lenstra, Spiess, Zarhin; see \cite{milne2001}, A.7).

\subsection{Statement of the standard conjectures}

Let $X$ be a smooth projective variety of dimension $n$ over an algebraically
closed field $k$ (possibly of characteristic zero). For $\ell\neq
\text{char}(k)$, let $L\colon H^{i}(X,\mathbb{Q}{}_{\ell})\rightarrow
H^{i+2}(X,\mathbb{Q}_{\ell})(1)$ denote the Lefschetz operator on $\ell$-adic
\'{e}tale cohomology defined by an ample divisor. According to the strong
Lefschetz theorem (\cite{deligne1980}), the map
\[
L^{n-2i}\colon H^{2i}(X,\mathbb{Q}_{\ell})(i)\rightarrow H^{2n-2i}%
(X,\mathbb{Q}{}_{\ell})(n-i)
\]
is an isomorphism for all $i\leq n/2$. Let $A_{\ell}^{i}(X)$ denote the
$\mathbb{Q}$-subspace of $H^{2i}(X,\mathbb{Q}_{\ell})(i)$ spanned by the
classes of algebraic cycles.

\begin{lconjecture}
The map
\[
L^{n-2i}\colon A_{\ell}^{i}(X)\rightarrow A_{\ell}^{n-i}(X)
\]
is an isomorphism for all $i\leq n/2$.
\end{lconjecture}

The map $L^{n-2i}$ is always injective, and it is surjective, for example, if
the folklore conjecture holds for $X$ and $\ell$ (because then $A_{\ell}%
^{i}(X)$ and $A_{\ell}^{n-i}(X)$ are finite dimensional and dual).

Assuming the Lefschetz standard conjecture, we get a decomposition%
\[
A_{\ell}^{i}(X)=P_{\ell}^{i}(X)\oplus LP_{\ell}^{i-1}(X)\oplus\cdots,
\]
where $P_{\ell}^{j}(X)=\Ker\left(  L^{n-2j+1}\colon A_{\ell}^{j}(X)\rightarrow
A_{\ell}^{n-j+1}(X)\right)  $.

\begin{hconjecture}
The pairing%
\begin{equation}
a,b\mapsto(-1)^{i}\langle L^{n-2i}a\cdot b\rangle\colon P_{\ell}^{i}(X)\times
P_{\ell}^{i}(X)\rightarrow\mathbb{Q} \label{j10}%
\end{equation}
is positive definite for all $i\leq n/2$.
\end{hconjecture}

In characteristic zero, the Hodge standard conjecture follows from Hodge theory.

\begin{theorem}
\label{j16}Let $X$ be a smooth projective variety over $\mathbb{F}$. If the
Tate and standard conjectures hold for $X$ and one $\ell\neq p$, then they
hold for $X$ and all $\ell\neq p$.
\end{theorem}

\begin{proof}
Suppose that the Tate and standard conjectures hold for $X$ and $\ell_{0}$.
The standard conjecture of Hodge type for $X$ and $\ell_{0}$ implies the
folklore conjecture for $X$ and $\ell_{0}$. Because the Tate and folklore
conjectures hold for $X$ and $\ell_{0}$, they hold for $X$ and all $\ell\neq
p$ (see \ref{j26}). Because the folklore conjecture is true for all $\ell\neq
p$, the standard conjectures are independent of $\ell$.
\end{proof}

The Lefschetz standard conjecture is known for abelian varieties (Kleiman,
Lieberman) --- we even know that the correspondence is given by a Lefschetz
class (\cite{milne1999lc}, 5.9). For the Hodge standard conjecture, the first
interesting case is abelian fourfolds. Concerning this, there is the following result.

\begin{theorem}
[Ancona]\label{j33}Let $A$ be an abelian fourfold over an algebraically closed
field $k$. The pairing (\ref{j10}) is positive definite on the algebraic
cycles modulo numerical equivalence.
\end{theorem}

\begin{proof}
This is the main theorem of \cite{ancona2021}.
\end{proof}

\begin{corollary}
\label{j34}Let $A$ be an abelian fourfold over $\mathbb{F}{}$. The Hodge
standard conjecture holds for $A$ and all $\ell\in s(A)$ (see \ref{j35}). If
the Tate conjecture holds for $A$ and one $\ell\in s(A)$, then the Hodge
standard conjecture holds for $A$ and all $\ell\neq p$.
\end{corollary}

\begin{proof}
The first assertion is obvious, and the second follows from Theorem \ref{j16}.
\end{proof}

\subsection{A criterion for the Hodge standard conjecture}

Let $k$ be an algebraically closed field, and let $\Mot(k;\mathcal{S}{})$ be
the category of motives modulo numerical equivalence generated by a collection
$\mathcal{S}$ of smooth projective varieties over $k$. Suppose that, for some
prime $\ell_{0}$, the Lefschetz standard conjecture and the folklore
conjecture hold for the varieties in $\mathcal{S}{}$. Then $\Mot(k;\mathcal{S}%
{})$ is a tannakian category with a fibre functor $\omega_{\ell_{0}}$, and it
has a natural structure of a Tate triple.

For every variety $X$ in $\mathcal{S}{}$ and $i\in\mathbb{N}{}$, there exists
a subobject $p^{i}(X)$ of $h^{2i}(X)(i)$ and a pairing $\phi^{i}\colon
p^{i}(X)\otimes p^{i}(X)\rightarrow\1$, both fixed by the fundamental group of
$\Mot(k)$, such that $\omega_{\ell_{0}}(\phi^{i})$ is the pairing in the
statement of the Hodge standard conjecture.

\begin{proposition}
\label{j19}The Hodge standard conjecture holds for the varieties in
$\mathcal{S}{}$ and $\ell_{0}$ if and only if there exists a polarization on
$\Mot(k;\mathcal{S}{})$ for which the forms $\phi^{i}$ are positive.
\end{proposition}

\begin{proof}
$\Rightarrow$: If the Hodge standard conjecture holds for all $X\in
\mathcal{S}{}$, then there is a canonical polarization $\Pi$ on
$\Mot(k;\mathcal{S}{})$ for which the bilinear forms
\[
\varphi^{i}\colon h^{i}(X)\otimes h^{i}(X)\overset{\id\otimes\ast
}{\rightarrow}h^{i}(X)\otimes h^{2n-i}(X)(n-i)\rightarrow h^{2n}%
(X)(n-i)\simeq\1(-i)
\]
are positive (Saavedra 1972, VI 4.4) --- here $X\in\mathcal{S}$ has dimension
$n$ and $\ast$ is defined by an ample divisor of $X$. The restriction of
$\varphi^{2i}\otimes\id_{\1(2i)}$ to the subobject $p^{i}(X)$ of
$h^{2i}(X)(i)$ is the form $\phi^{i}$, which is therefore positive for $\Pi$
(\cite{deligneM1982}, 4.11b).

$\Leftarrow$: We let $\mathsf{V}$ denote the category of $\mathbb{Z}$-graded
$\mathbb{C}{}$-vector spaces $V$ equipped with a semilinear automorphism $a$
such that $a^{2}v=(-1)^{n}v$ $(v\in V^{n})$; it has a natural structure of a
Tate triple over $\mathbb{R}{}$ (\cite{deligneM1982}, 5.3). Let $\Pi$ be a
polarization on $\Mot(k;\mathcal{S}{})$ for which the forms $\phi^{i}$ are
positive. There exists a morphism of Tate triples $\xi\colon\Mot(k;\mathcal{S}%
{})\rightarrow\mathsf{V}$ such that $\xi$ maps $\Pi$ to the canonical
polarization $\Pi^{V}$ on $\mathsf{V}$; in particular, for $X$ of weight $0$
and $\phi\in\Pi(X)$, $(\gamma^{V}\circ\xi)(\phi)$ is a positive definite
symmetric form on $(\gamma^{V}\circ\xi)(X)$ (ibid.~5.20). The restriction of
$\gamma^{V}\circ\xi$ to $\Mot(k;\mathcal{S}{})^{\boldsymbol{\pi}}$ is
(uniquely) isomorphic to $\gamma^{\text{Mot}}$, and so $\gamma^{\text{Mot}%
}(\phi^{i})$ is positive definite.
\end{proof}

Let $A$ be an abelian variety over $\mathbb{F}$, and let $\langle
A\rangle^{\otimes}$ be the category of motives modulo numerical equivalence
generated by $A$ and $\mathbb{P}^{1}$; it has a natural structure of a Tate
triple. Note that $\langle A\rangle^{\otimes}$ contains the motives of the
powers of $A$. Recall that the Lefschetz standard conjecture holds for abelian varieties.

\begin{corollary}
\label{j36}Let $\ell\in s(A)$; the Hodge standard conjecture holds for $\ell$
and the powers of $A$ if and only if there exists a polarization on the Tate
triple $\langle A\rangle^{\otimes}$ such that the forms $\phi^{i}(A^{r})\colon
p^{i}(A^{r})\otimes p^{i}(A^{r})\rightarrow\1$ are positive, all
$i,r\in\mathbb{N}{}$.
\end{corollary}

\begin{proof}
Immediate consequence of the proposition.
\end{proof}

\begin{aside}
\label{j37}Besides abelian varieties, the most natural varieties to
\textquotedblleft test\textquotedblright\ these conjectures on are $K3$
surfaces. The Hodge conjecture is known for squares of $K3$ surfaces with
complex multiplication (Mukai; \cite{buskin2019}, 1.3), and the Tate
conjecture is known for $K3$ surfaces over $\mathbb{F}{}$ and their squares
(many authors; \cite{itoIT2021}).
\end{aside}

\section{Mixed characteristic}

Let $\mathbb{Q}^{\mathrm{al}}$ be the algebraic closure of $\mathbb{Q}{}$ in
$\mathbb{C}{}$, and let $w_{0}$ be a prime of $\mathbb{Q}{}^{\mathrm{al}}$
dividing $p$.\footnote{In the following, it is possible to replace
$\mathbb{Q}{}^{\mathrm{al}}$ with an algebraic closure $C$ of its completion
at $w_{0}$, and then choose an embedding of $C$ into $\mathbb{C}{}$, but this
requires the axiom of choice instead of the more benign axiom of dependent
choice.} The residue field at $w_{0}$ is an algebraic closure $\mathbb{F}$ of
$\mathbb{F}_{p}$. Let $A$ be an abelian variety over $\mathbb{Q}%
{}^{\mathrm{al}}$ with good reduction to an abelian variety $A_{0}$ over
$\mathbb{F}{}$. If $A$ is of CM-type, then there is a unique homomorphism
$L(A_{0})\rightarrow L(A)$ compatible with the actions of the groups on
cohomology and with the specialization isomorphisms $H^{i}(A^{s},\mathbb{Q}%
{}_{\ell})\simeq H^{i}(A_{0}^{s},\mathbb{Q}{}_{\ell})$. As absolute Hodge
classes are Tate classes, some Frobenius endomorphism for $A_{0}$ will lie in
$\MT(A)(\mathbb{Q}{})$, and so $P(A_{0})\subset\MT(A)$. Thus we have a
commutative diagram%
\begin{equation}
\begin{tikzcd} \MT(A)\arrow[hook]{r}&L(A)\\ P(A_{0})\arrow[hook]{r}\arrow[hook]{u}&L(A_{0})\arrow[hook]{u}. \end{tikzcd} \label{jbb}%
\end{equation}
If $A$ is not CM, the diagram still exists, but only as a diagram of groups in
the tannakian category of Lefschetz motives generated by $A_{0}$.

\begin{theorem}
\label{j11}Let $A$ be a CM abelian variety over $\mathbb{Q}{}^{\mathrm{al}}$
with good reduction at $w_{0}$ to an abelian variety $A_{0}$ over
$\mathbb{F}{}$. Assume that%
\begin{equation}
P(A_{0})=L(A_{0})\cap\MT(A)\quad\quad\text{(intersection inside }L(A)).
\tag{*}\label{blah}%
\end{equation}
If the Hodge conjecture holds for $A$ and its powers, then

\begin{enumerate}
\item the Tate and folklore conjectures hold for $A_{0}$ and its powers;

\item the Hodge standard conjecture holds for $A_{0}$ and its powers.
\end{enumerate}
\end{theorem}

\noindent For any $A$, there exists a CM abelian variety $B$ such that
$A\times B$ satisfies (*) (see \ref{j12} below).

\begin{proof}
(a) Let $\ell_{0}\in s(A_{0})$ (so the folklore conjecture holds for $\ell
_{0}$ and the powers of $A_{0}$). Let $M(A_{0})$ be the algebraic subgroup of
$L(A_{0})$ fixing the $\ell_{0}$-adic algebraic classes on the powers of
$A_{0}$. Because algebraic classes are Tate, $P(A_{0})\subset M(A_{0})$, and
the assumption of the Hodge conjecture implies that $M(A_{0})\subset\MT(A)$.
Now (*) implies that $M(A_{0})=P(A_{0})$, and so the $\ell_{0}$-adic Tate
classes on the powers of $A_{0}$ are algebraic.\footnote{This uses that the
cohomology classes fixed by $M(A)$ are algebraic --- see the Appendix}
According to Theorem \ref{j26}, this implies that the Tate and folklore
conjectures hold for the powers of $A_{0}$ and all $\ell\neq p$.

(b) We defer the proof to later in this section.
\end{proof}

\begin{corollary}
\label{j17}Let $A$ be an abelian variety over $\mathbb{Q}{}^{\mathrm{al}}$
with good reduction at $w_{0}$ to an abelian variety $A_{0}$ over
$\mathbb{F}{}$.

\begin{enumerate}
\item If $A$ is neat, then the Hodge conjecture holds for the powers of $A$.

\item If $A_{0}$ is neat, then (\ref{blah}) holds for $A$ and the Tate
conjecture holds for the powers of $A_{0}$.

\item If $A_{0}$ is neat and $\End^{0}(A)=\End^{0}(A_{0})$, then $A$ is neat.

\item If both $A$ and $A_{0}$ are neat, then the Hodge standard conjecture for
the powers of $A_{0}$.
\end{enumerate}
\end{corollary}

\begin{proof}
Statements (a) and (b) are true almost by definition, (c) follows from the
diagram (\ref{jbb}), and (d) follows from (b) of Theorem \ref{j11} and the
next remark.
\end{proof}

\begin{remark}
\label{j73}In Theorem \ref{j11}, it is not necessary to assume that the Hodge
classes on $A$ are algebraic, only almost-algebraic. Part (a) of the theorem
also holds for non-CM abelian varieties $A$ with essentially the same proof.
\end{remark}

\begin{remark}
If $L(A_{0})=L(A)$, then (\ref{blah}) holds if an only if $\pi_{A_{0}}$
generates $\MT(A)$. Here $\pi_{A_{0}}$ denotes a sufficiently high power of
the Frobenius element defined by a some model of $A_{0}$ over a finite field.
\end{remark}

\qquad We list some abelian varieties for which (\ref{blah}) holds.

\begin{example}
\label{j12}Let $K$ be a CM-subfield of $\mathbb{Q}^{\mathrm{al}}{}$, finite
and Galois over $\mathbb{Q}{}$, and let $A$ be a CM abelian variety over
$\mathbb{Q}{}^{\mathrm{al}}$ with reflex field contained in $K$ and such that
every simple CM abelian variety over $\mathbb{Q}{}^{\mathrm{al}}$ with reflex
field contained in $K$ is an isogeny factor of $A$. Then (\ref{blah})
holds\footnote{For example, for each CM-type $\Phi$ on $K$, let $A_{\Phi}$ be
an abelian variety of $\mathbb{Q}{}^{\mathrm{al}}$ of CM-type $(K,\Phi)$. Then
$A\overset{\df}{=}\prod A_{\Phi}$ has the required property.}

In this case, (\ref{blah}) becomes the formula $P^{K}=L^{K}\cap S^{K}$ of
\cite{milne1999lm}, Theorem 6.1. There it is proved under the hypothesis that
$K$ contains an imaginary quadratic field in which $p$ splits, but this
assumption is unnecessary (see the Appendix).
\end{example}

\begin{example}
\label{j13}Let $(A\times B^{m-2},Q)$ be as in Example \ref{j77}. Let $K$ be
the subfield of $\mathbb{Q}{}^{\mathrm{al}}$ generated by the conjugates of
$E$ in $\mathbb{Q}{}^{\mathrm{al}}$. Let $H=\Gal(K/\varphi_{0}E)$ and let
$D(w_{0})\subset\Gal(K/\mathbb{Q}{})$ be the decomposition group of $w_{0}|K$.
Assume that $p$ splits in $Q$ and that
\[
H\cdot D(w_{0})=D(w_{0})\cdot H.
\]
Then (\ref{blah}) holds for $A\times B$ (see \cite{milne2001}). Thus Theorem
\ref{j11} applies to $A\times B$.
\end{example}

\begin{example}
\label{j62}We give an example where (*) fails. Let $\pi$ be a Weil $q$-integer
of degree $6$ with the following properties:

\begin{enumerate}
\item for all $v|p$, $\frac{v(\pi)}{v(q)}[\mathbb{Q}{}[\pi]_{v}\colon
\mathbb{Q}{}_{p}]\equiv0$ mod $1;$

\item there exists an imaginary quadratic field $Q\subset\mathbb{Q}{}[\pi]$
such that $\Nm_{\mathbb{Q}{}[\pi]/Q}(\pi^{2}/q)=1$.
\end{enumerate}

\noindent Such a $\pi$ exists. The simple abelian variety $A_{0}$ over
$\mathbb{F}{}_{q}$ with Weil integer $\pi$ has dimension $3$ and endomorphism
algebra $E\overset{\df}{=}\mathbb{Q}{}[\pi]$. Up to isogeny, there exists a
lift $A$ of $A_{0}$ to $\mathbb{Q}{}^{\mathrm{al}}$ with complex
multiplication by $E$ (\cite{tate1968}, Thm 2). As $A$ has dimension $3$, it
is neat. On the other hand,%
\[
P(A_{0})\subset\Ker(L(A_{0})\rightarrow(\mathbb{G}_{m})_{Q/\mathbb{Q}{}})
\]
and so $P(A_{0})\neq L(A_{0})\cap\MT(A)$. The group $P(A_{0})$ acts trivially
on
\[
\big(  \bigwedge\nolimits_{Q\otimes\mathbb{Q}_{\ell}}^{6}H^{1}(A^{2}%
,\mathbb{Q}{}_{\ell})\big)  (3),
\]
which therefore consists of exotic Tate classes.
\end{example}

\subsection{Proof of (b) of Theorem \ref{j11}}

Let $R\colon\mathsf{C}_{1}\rightarrow\mathsf{C}_{2}$ be a functor of Tate
triples. We say that $R$ maps a polarization $\Pi_{1}$ on $\mathsf{C}_{1}$ to
a polarization $\Pi_{2}$ on $\mathsf{C}_{2}$ if $\psi\in\Pi_{1}(X)$ implies
$R(\psi)\in\Pi_{2}(R(X))$, i.e., if all Weil forms positive for $\Pi_{1}$ map
to Weil forms positive for $\Pi_{2}$.

Let $A$ be an abelian variety over $\mathbb{Q}{}^{\mathrm{al}}$ with good
reduction at $w_{0}$ to an abelian variety $A_{0}$ over $\mathbb{F}{}$. Assume
that (\ref{blah}) holds and that the Hodge conjecture holds for $A$ and its
powers. According to Theorem \ref{j16}, it suffices to prove the Hodge
standard conjecture for a single $\ell$. Implicitly, we choose $\ell\in s(A)$.

Let $\langle A\rangle^{\otimes}$ be the category of motives generated by $A$
and $\mathbb{P}{}^{1}$ using Hodge classes as correspondences, and let
$\langle A_{0}\rangle^{\otimes}$ be the category of numerical motives
generated by $A_{0}$ and $\mathbb{P}{}^{1}$. We regard both as Tate triples
with their natural additional structures. Because we are assuming the Hodge
conjecture for the powers of $A$, there is a reduction functor $q\colon\langle
A\rangle^{\otimes}\rightarrow\langle A_{0}\rangle^{\otimes}$. This realizes
$\langle A_{0}\rangle^{\otimes}$ as a quotient of $\langle A\rangle^{\otimes}$
in the sense of \cite{milne2007qtc}, 2.2. Hodge theory provides the Tate
triple $\langle A\rangle^{\otimes}$ with a canonical polarization $\Pi$.

\begin{lemma}
\label{j81} To prove (b) of Theorem \ref{j11}, it suffices to show that there
exists a polarization $\Pi_{0}$ on $\langle A_{0}\rangle^{\otimes}$ such that
$q$ maps $\Pi$ to $\Pi_{0}$.
\end{lemma}

\begin{proof}
Let $D$ be an ample divisor on $A$, and use it to define the bilinear forms
\[
\phi^{i}(A^{r})\colon p^{i}(A^{r})\otimes p^{i}(A^{r})\rightarrow\1,\quad
i,r\in\mathbb{N}{}.
\]
Then $\phi^{i}(A^{r})$ is positive (by definition) for $\Pi$, and $q(\phi
^{i}(A^{r}))$ is the bilinear form $\phi^{i}(A_{0}^{r})$ defined by the ample
divisor $D_{0}$ on $A_{0}$. As $q$ maps $\Pi$ to $\Pi_{0}$, the forms
$\phi_{0}^{i}(A^{r})$ are positive for $\Pi_{0}$, which implies that the Hodge
standard conjecture holds for the varieties in $\langle A_{0}\rangle^{\otimes
}$ (see \ref{j36}).
\end{proof}

\begin{lemma}
\label{j82}There exists a polarization $\Pi_{0}$ on $\langle A_{0}%
\rangle^{\otimes}$ such that $q$ maps $\Pi$ to $\Pi_{0}$ if there exists an
$X\in\left(  \langle A\rangle^{\otimes}\right)  ^{P(A_{0})}$ and a $\psi\in
\Pi(X)$ such that $\MT(A)/P(A_{0})$ acts faithfully on $X$ and $q(\psi)$ is a
positive definite form on the vector space $q(X)$.
\end{lemma}

\begin{proof}
Apply \cite{milne2002p}, 1.5.
\end{proof}

We now prove (b) of Theorem \ref{j11} by showing that there exists a pair
$(X,\psi)$ satisfying the conditions of \ref{j82}. Let $X=\underline{\End}%
(h_{1}A)^{P(A_{0})}$. In Lemma \ref{j83} below, we show that $L(A)/L(A_{0})$
acts faithfully on $\underline{\End}(h_{1}A)^{L(A_{0})}$. As (\ref{blah})
implies that $\MT(A)/P(A_{0})$ injects into $L(A)/L(A_{0})$, this shows that
$\MT(A)/P(A_{0})$ acts faithfully on $X$.

Let $\phi\colon h_{1}A\otimes h_{1}A\rightarrow\1(-1)$ be the form defined by
an ample divisor $D$ on $A$, and let $\psi$ be the symmetric bilinear form on
$\End(h_{1}A)$ defined by $\phi$ (see \cite{milne2002p}, 1.1). Write $\psi|$
for the restriction of $\psi$ to $X$. Then $\psi|\in\Pi(X)$ and it remains to
show that $q(\psi|)$ is positive definite. But $q(X)=\End^{0}(A_{0})$ and
$q(\psi)$ is the trace pairing $u,v\mapsto\Tr(u\cdot v^{\dagger})$ of the
Rosati involution defined by the divisor $D_{0}$ on $A_{0}$, which is positive
definite by \cite{weil1948}, Th\'{e}or\`{e}me 38.

\begin{lemma}
\label{j83}Let $A$ be a CM abelian variety over $\mathbb{Q}{}^{\mathrm{al}}$
with good reduction at $w_{0}$ to an abelian variety $A_{0}$ over
$\mathbb{F}{}$. The action of $L(A)/L(A_{0})$ on $\underline{\mathrm{End}%
}(h_{1}A)^{L(A_{0})}$ is faithful.
\end{lemma}

\begin{proof}
First an elementary remark. Let $T$ be a torus acting on a finite dimensional
vector space $V$, and let $L$ be a subtorus of $T$. Let $\chi_{1},\ldots
,\chi_{n}$ be the characters of $T$ occurring in $V$. Then $T$ acts faithfully
on $V$ if and only if $\chi_{1},\ldots,\chi_{n}$ span $X^{\ast}(T)$ as a
$\mathbb{Z}{}$-module --- assume this. The characters of $T$ occurring in
$\mathrm{End}(V)$ are $\{\chi_{i}-\chi_{j}\}$, and the set of those occurring
in $\mathrm{End}(V)^{L}$ is
\begin{equation}
\{\chi_{i}-\chi_{j}\mid\chi_{i}|L=\chi_{j}|L\}. \tag{*}%
\end{equation}
On the other hand,
\begin{equation}
X^{\ast}(T/L)=\left\{  \sum a_{i}\chi_{i}\,\middle|\,\sum a_{i}\chi
_{i}|L=0\right\}  . \tag{**}%
\end{equation}
Thus, $T/L$ will act faithfully on $\mathrm{End}(V)^{L}$ if the set (*) spans
the $\mathbb{Z}{}$-module (**).

It suffices to prove the lemma for a simple $A$. Let $K$ be a (sufficiently
large) CM field, Galois over $\mathbb{Q}{}$, with Galois group $\Gamma$. Then
$A$ corresponds to a $\Gamma$-orbit $\Psi$ of CM-types on $K$
(\cite{milne1999lm}, 2.3). Each $\psi\in\Psi$ defines a character of $L(A)$,
and $L(A)$ acts on $\omega_{\ell}(h_{1}A)$ through the $\psi$ in $\Psi$. For
$\psi\in\Psi$, let $\pi(\psi)$ denote the Weil germ attached to $\psi$ by the
Taniyama formula (\cite{tate1966e}, Thm 5).

The abelian variety $A_{0}$ is a isotypic${}$, and hence corresponds to a
$\Gamma$-orbit $\Pi$ of Weil germs (\cite{milne1999lm}, 4.1). In fact,
$\Pi=\{\pi(\psi)\mid\psi\in\Psi\}$ (ibid.~5.4). Each $\pi\in\Pi$ defines a
character of $X^{\ast}(L(A_{0})),$ and the homomorphism $X^{\ast
}(L(A))\rightarrow X^{\ast}(L(A_{0}))$ sends $\psi$ to $\pi(\psi)$. The
elements of $\Psi$ can be numbered $\psi_{1},\ldots,\psi_{n},\bar{\psi}%
_{n+1},\ldots,\bar{\psi}_{2n}$, in such a way that $\pi(\psi_{1})=\cdots
=\pi(\psi_{d})=\pi_{1}$, $\pi(\psi_{d+1})=\cdots=\pi(\psi_{2d})=\pi_{2},$
etc.. Now $\sum a_{i}\cdot\psi_{i}|L(A_{0})=\sum a_{i}\cdot\pi(\psi_{i})$,
which is zero if and only if $\sum_{i=1}^{d}a_{i}=0$, $\sum_{i=d+1}^{2d}%
a_{i}=0$, etc. But then
\[
\sum_{i=1}^{n}a_{i}\psi_{i}=\sum_{i=1}^{d}a_{i}(\psi_{i}-\psi_{1})+\cdots,
\]
which (by the remark) shows that
\[
S(A)/S(A_{0})\simeq L(A)/L(A_{0})
\]
acts faithfully on $\underline{\mathrm{End}}(h_{1}A)^{L(A_{0})}$.
\end{proof}

\begin{aside}
\label{j93}Strictly, the proof only shows that the Hodge standard conjecture
holds for a Lefschetz operator on $A_{0}$ coming from $A$. However, recall the
following theorem. Let $(A,\lambda)$ be a polarized abelian variety over
$\mathbb{F}$. For some discrete valuation ring $R$ containing the ring of Witt
vectors $W(\mathbb{F}{}){}$ and finite over $W(\mathbb{F}{})$, there exists a
polarized abelian scheme $(B,\mu)$ over $R$ whose generic fibre has complex
multiplication and whose special fibre is isogenous to $(A,\lambda)$. Indeed,
Mumford 1970, Corollary 1, p.~234, allows us to assume that the polarization
$\lambda$ is principal, in which case we can apply Zink 1983, 2.7 (with
$L=\mathbb{Q}$).
\end{aside}

\subsection{Proof of Theorem \ref{j0}}

Let $A$ and $B$ be as in the statement of Theorem \ref{j0}. As $A$ is neat but
$A\times B$ is not neat, there exists an embedding of $Q\overset{\df}{=}%
\End^{0}(B)$ into $\End^{0}(A)$ and an $m\in\mathbb{N}{}$ such $(A\times
B^{m},Q)$ is of Weil type and%
\[
1\rightarrow\Hg(A\times B)\longrightarrow S(A\times B)\overset{\rho
}{\longrightarrow}U_{Q}\rightarrow1
\]
is exact (\ref{j72}). Thus, the Hodge conjecture holds for the powers of
$A_{\mathbb{C}{}}\times B_{\mathbb{C}{}}$ if the Weil classes are algebraic
(Theorem \ref{j63}).

As $A$ and $B$ are neat, they satisfy the Mumford--Tate conjecture, and
therefore $A\times B$ does too. Thus, the Tate conjecture for the powers of
$A\times B$ follows from the Hodge conjecture for the powers of $(A\times
B)_{\mathbb{C}{}}$ (Theorem \ref{j92}). 

As $A_{0}$ is neat but $A_{0}\times B_{0}$ is not neat, $B_{0}$ must be
ordinary. Now the same argument as in the proof of Proposition \ref{j72} gives
us an exact sequence%
\[
1\rightarrow P^{\prime}(A\times B)\rightarrow S(A\times B)\rightarrow
U\rightarrow1
\]
in which $U$ is one-dimensional and the projection $S(B)\rightarrow U_{0}$ is
surjective. Thus, we have a commutative diagram with exact rows,%
\[
\begin{tikzcd}
1\arrow{r}  &\Hg(A\times B)\arrow{r}& S(A)\times S(B)\arrow{r}{\rho} & U_Q\arrow{r}&1\\
1\arrow{r}
&P^{\prime}(A_{0}\times B_0)\arrow{r}\arrow{u}
&S(A_{0})\times S(B_0)\arrow{r}\arrow{u}
&U_{0}\arrow{r}\arrow{u} &1.
\end{tikzcd}
\]
The maps $S(B_{0})\rightarrow S(B)\rightarrow U$ are isomorphisms. A diagram
chase now shows that $U_{0}\rightarrow U$ is injective, and so
\[
P^{\prime}(A_{0}\times B_{0})=S(A_{0}\times B_{0})\cap\Hg(A\times B).
\]
Thus (*) holds. As the Hodge conjecture holds for the powers of $A\times B$,
the Tate, folklore, and Hodge standard conjecture hold for the powers of
$A_{0}\times B_{0}$ (Theorem \ref{j11}).

\section{Some examples}

We list a few examples.

\subsection{Products of elliptic curves}

\begin{theorem}
\label{j56}Let $A$ be a product of elliptic curves over an algebraically
closed field $k$.

(a) If $k$ has characteristic zero, then $A$ is neat; in particular, it
satisfies the Hodge conjecture.

(b) If $k$ is the algebraic closure of a field finitely generated over
$\mathbb{Q}$, then $A$ is neat; in particular, it satisfies the Tate conjecture.

(c) If $k=\mathbb{F}$, then $A$ is neat; in particular, it satisfies the Tate conjecture.

(d) For any field $k$, $A$ satisfies the Hodge standard conjecture.
\end{theorem}

\begin{proof}
(a) Let $A=A_{1}\times\cdots\times A_{m}$. Suppose first that the $A_{i}$ are
CM. Choose a prime number $p$ such that the $A_{i}$ have good ordinary
reduction at $p$. Then the statement follows from (c) and \ref{j17}(c).

As elliptic curves are neat, in the general case it suffices to prove that, if
$A_{1},\ldots,A_{m}$ are elliptic curves no two of which are isogenous, then%
\begin{equation}
\Hg(A_{1}\times\cdots\times A_{m})=\Hg(A_{1})\times\cdots\times\Hg(A_{m}).
\label{j99}%
\end{equation}
This can be proved using Goursat's lemma (\cite{imai1976}, p.~367).

(b) As the Mumford--Tate conjecture is known in this case, the statement
follows from (a).

(c) This is the main theorem of \cite{spiess1999}.

(d) By a specialization argument, it suffices to prove this over $\mathbb{F}%
{}$. Given a product of elliptic curves over $\mathbb{F}{}$, we may lift it to
a product of CM elliptic curves over $\mathbb{Q}^{\mathrm{al}},$ and apply
\ref{j17}(d).
\end{proof}

\subsection{Abelian surfaces}

Abelian surfaces are neat, both over $\mathbb{F}$ and over fields of
characteristic zero. Thus, in characteristic zero, powers of abelian surfaces
satisfy the Hodge and Tate conjectures, and, in characteristic $p$, they
satisfy the Tate and standard conjectures (see Corollary \ref{j17}).

\subsection{Abelian threefolds}

In characteristic zero, abelian threefolds are neat. Over $\mathbb{F}$,
abelian threefolds satisfy the Tate conjecture, and they are neat except in
the following case: $E\overset{\df}{=}\End^{0}(A_{0})$ is a field containing
an imaginary quadratic field $Q$ and $\Nm_{E/Q}(\pi_{A}^{2}/q)=1$. See
\cite{zarhin1994}.

Let $A_{0}$ be an exceptional abelian threefold over $\mathbb{F}{}$, and let
$B_{0}$ be an elliptic curve over $\mathbb{F}{}$ with $Q=\End^{0}(B_{0})$.
There exist lifts $A$ and $B$ of $A_{0}$ and $B_{0}$ (up to isogeny) as in
1.12. Therefore, the Tate and standard conjectures hold for the varieties
$A^{r}\times B^{s}$ (see 3.6).

\subsection{Products of threefolds and elliptic curves over $\mathbb{F}{}$}

\begin{theorem}
\label{j90}Let $A_{0}$ (resp.~$B_{0})$ be a simple abelian threefold
(resp.~elliptic curve) over $\mathbb{F}{}$. If $A_{0}$ is neat, then the Tate
and standard conjectures hold for the abelian varieties $A_{0}^{r}\times
B_{0}^{s}$ ($r,s\in\mathbb{N}{})$.
\end{theorem}

\begin{proof}
Lift $A_{0}$ and $B_{0}$ (up to isogeny) to abelian varieties $A$ and $B$ over
$\mathbb{Q}{}^{\mathrm{al}}$. If neither $A\times B$ nor $A_{0}\times B_{0}$
is neat, then we can apply Theorem \ref{j0} to prove the statement. If $A\times
B$ and $A_{0}\times B_{0}$ are both neat, we can apply Corollary \ref{j17}.
We leave the remaining cases as an exercise.
\end{proof}

\subsection{Simple abelian fourfolds}

\begin{theorem}
\label{j91}Let $A$ be a CM abelian fourfold over $\mathbb{Q}^{\mathrm{al}}$
with good reduction at $w_{0}$ to an abelian variety $A_{0}$ over
$\mathbb{F}{}$. Suppose that $\End^{0}(A_{0})=\mathbb{Q}[\pi_{A}]$ and that
$\pi_{A}$ generates $\MT(A)$. If either

\begin{enumerate}
\item $A$ is not of Weil type, or

\item $A$ is of Weil type relative to $Q\subset\mathbb{Q}{}[\pi_{A}]$ and
$\det(A,Q,\lambda)=1$ for some polarization $\lambda$,
\end{enumerate}

\noindent then the Hodge conjecture holds for the powers of $A_{\mathbb{C}}$ and the Tate
and standard conjectures hold for the powers of $A_{0}$.
\end{theorem}

\begin{proof}
The hypotheses imply that $L(A_{0})=L(A)$ and that (\ref{blah}) holds. Under
either (a) or (b), the Hodge conjecture holds for the powers of $A$, and so we
can apply Theorem \ref{j11}.
\end{proof}

\subsection{General abelian varieties of Weil type}

Let $(A,Q,\lambda)$ be a general abelian variety of Weil type over
$\mathbb{Q}^{\mathrm{al}}$. We saw in Theorem \ref{j52} that the Hodge
conjecture holds for the powers of $A$ if the Weil classes on $(A,Q)$ are
algebraic. I expect that if $A$ has good reduction to an abelian variety
$A_{0}$ over $\mathbb{F}{}$, then (*) holds and the Tate and standard
conjectures hold for the powers of $A_{0}$ under the same hypothesis on the
Weil classes. Similar statements should hold for products of general abelian
varieties of Weil type (see \ref{j80}).

\appendix

\section{Appendix}

We add some commentary and details.

\subsection{Proof of the tannakian property of $M(A)$}

Let $A$ be an abelian variety over $\mathbb{C}{}$, and let $M(A)$ denote the
algebraic subgroup of $L(A)$ fixing the algebraic classes in
$H\overset{\df}{=}\bigoplus_{r,s}H^{2r}(A^{s},\mathbb{Q}{})(r)$. Then every
element of $H$ fixed by $M(A)$ is algebraic.

In the 1980s, this statement was considered to be beyond reach; in 1990s, it
was considered obvious. What changed was that Jannsen proved that the ring of
correspondences modulo numerical equivalence is semisimple (see below).

Before proving the statement, I illustrate its importance. Let $A$ be an
abelian variety, and let $G$ by the subgroup of $\GL_{V(A)}$ fixing some
algebraic classes on $A$. Then \textit{every} cohomology class fixed by $G$ is
algebraic. This follows from the fact that $G$ (obviously) contains $M(A)$.
Note that this application does not require us to know $M(A)$.

Let $\Mot(\mathbb{C})$ denote the category of abelian motives over
$\mathbb{C}{}$ modulo numerical equivalence. This is a pseudo-abelian tensor
category, and Jannsen's result shows that it is abelian. As the folklore
conjecture is known for abelian varieties over $\mathbb{C}{}$, there is a
Betti fibre functor $\omega_{B}$ on $\Mot(\mathbb{C})$. Therefore
$\Mot(\mathbb{C})$ is a neutral tannakian category. Let $\langle
A\rangle^{\otimes}$ denote the tannakian subcategory of $\langle
A\rangle^{\otimes}$ generated by $A$ and $\mathbb{P}{}^{1}$. Then
$M(A)=\underline{\Aut}^{\otimes}(\omega_{B}|\langle A\rangle^{\otimes})$, and
the above statement is simply an expression of tannakian duality.

The same statement holds over $\mathbb{F}$ when we replace Betti cohomology
with $\ell$-adic cohomology with $\ell\in s(A)$ (see \ref{j35}).

\subsection{Proof that the ring of correspondences is semisimple}

We prove that the existence of a Weil cohomology theory implies that the ring
of correspondences modulo numerical equivalence of an algebraic variety is semisimple.

Theorem \ref{b4} below is extracted from \cite{jannsen1992}. Throughout, $H$
is a Weil cohomology theory with coefficient field $Q$. We let $A_{H}^{r}(X)$
denote the $\mathbb{Q}{}$-subspace of $H^{2r}(X)(r)$ spanned by the algebraic
classes, and $A_{\mathrm{num}}^{r}(X)$ the quotient of $A_{H}^{r}(X)$ be the
left kernel of the intersection pairing%
\[
A_{H}^{r}(X)\times A_{H}^{d-r}(X)\rightarrow A_{H}^{d}(X)\simeq\mathbb{Q}.
\]

\begin{plain}
\label{b1}The $\mathbb{Q}{}$-vector space $A_{\mathrm{num}}^{r}(X)$ is
finite-dimensional over $\mathbb{Q}{}$: if $f_{1},\ldots,f_{s}\in A_{H}%
^{d-r}(X)$ span the subspace $Q\cdot A_{H}^{d-r}(X)$ of $H^{2d-2r}(X)(d-r)$,
then the map%
\[
x\mapsto(x\cdot f_{1},\ldots,x\cdot f_{s})\colon A_{H}^{r}(X)\rightarrow
\mathbb{Q}{}^{s}%
\]
has image $A_{\mathrm{num}}^{d-r}(X)$.
\end{plain}

\begin{plain}
\label{b2}Let $A_{H}^{r}(X,Q)=Q\cdot A_{H}$. Define $A_{\mathrm{num}}%
^{r}(X,Q)$ to be the quotient of $A_{H}^{r}(X,Q)$ by the left kernel of the
pairing
\[
A_{H}^{r}(X,Q)\times A_{H}^{d-r}(X,Q)\rightarrow A_{H}^{d}(X,Q)\simeq Q
\]
induced by cup product. Then $A_{H}^{r}(X)\rightarrow A_{\mathrm{num}}%
^{r}(X,Q)$ factors through $A_{\mathrm{num}}^{r}(X)$,%
\[
\begin{tikzcd}
A_{H}^{r}(X) \arrow{r}\arrow{d} &A_{\mathrm{num}}^{r}(X)\arrow[dashed]{d}\\
A_{H}^{r}(X,Q) \arrow{r} &A_{\mathrm{num}}^{r}(X,Q),
\end{tikzcd}
\]
and I claim that the map
\[
Q\otimes A_{\mathrm{num}}^{r}(X)\rightarrow A_{\mathrm{num}}^{r}(X,Q)
\]
is an isomorphism. As $A_{\mathrm{num}}^{r}(X,Q)$ is spanned by the image of
$A_{H}^{r}(X)$, the map is obviously surjective. Let $e_{1},\ldots,e_{m}$ be a
$\mathbb{Q}{}$-basis for $A_{\mathrm{num}}^{r}(X)$, and let $f_{1}%
,\ldots,f_{m}$ be the dual basis in $A_{\mathrm{num}}^{d-r}(X)$. If
$\sum_{i=1}^{m}a_{i}e_{i}$, $a_{i}\in Q$, is zero in $A_{\mathrm{num}}%
^{r}(X,Q)$, then $a_{j}=(\sum a_{i}e_{i})\cdot f_{j}=0$ for all $j$.
\end{plain}

\begin{plain}
\label{b3}Recall that the \emph{radical} $R(A)$ of a ring $A$ is the
intersection of the maximal left ideals in $A$. Equivalently it is the
intersection of the annihilators of simple $A$-modules. It is a two-sided
ideal in $A$. Every left (or right) nil ideal\footnote{An ideal is nil if all
of its elements are nilpotent. A finitely generated nil ideal is nilpotent.}
is contained in $R{}(A)$. For any ideal $\mathfrak{a}{}$ of $A$ contained in
$R{}(A)$, $R{}(A/\mathfrak{a}{})=R{}(A)/\mathfrak{a}{}$. The radical of an
artinian ring $A$ is nilpotent, and it is the largest nilpotent two-sided
ideal in $A$. (Bourbaki, A, VIII, \S 6.)
\end{plain}

\begin{theorem}
[Jannsen]\label{b4}Let $X$ be a smooth projective variety over a field $k$.
Then the $\mathbb{Q}$-algebra $A_{\mathrm{num}}^{\ast}(X\times X)$ is semisimple.
\end{theorem}

\begin{proof}
Let $B=A_{\mathrm{num}}^{\ast}(X\times X)$. Recall that $B$ has finite
dimension over $\mathbb{Q}{}$, and that multiplication in $B$ is composition
$\circ$ of correspondences. By definition of numerical equivalence, the
pairing%
\[
f,g\mapsto\langle f\cdot g\rangle\colon B\times B\rightarrow\mathbb{Q}{}%
\]
is nondegenerate. Let $f$ be an element of the radical $R{}(B)$ of $B$. We
have to show that $\langle f\cdot g\rangle=0$ for all $g\in B$.

Let $H$ be a Weil cohomology with coefficient field $Q$. Let $A=A_{H}^{\ast
}(X\times X,Q)$; then $A$ is a finite-dimensional $Q$-algebra, and there is a
surjective homomorphism
\[
A\overset{\df}{=}A_{H}^{d}(X\times X,Q)\rightarrow A_{\mathrm{num}}%
^{d}(X\times X,Q)\simeq Q\otimes B
\]
(see \ref{b2}). This maps the radical of $A$ onto that of $Q\otimes B$ (see
\ref{b3}). Therefore, there exists an $f^{\prime}\in R{}(A)$ mapping to
$1\otimes f$. For all $g\in A$,%
\begin{equation}
\langle f^{\prime}\cdot g^{t}\rangle=\sum\nolimits_{i}(-1)^{i}\Tr(f^{\prime
}\circ g\mid H^{i}(X)) \label{eq6}%
\end{equation}
(\cite{kleiman1968}, 1.3.6). As $f^{\prime}\circ g$ lies in $R{}(A)$, it is
nilpotent (see \ref{b3}), and so (\ref{eq6}) shows that $\langle f^{\prime
}\cdot g^{t}\rangle=0$.
\end{proof}

\subsection{Goursat's Lemma}

Let $B_{1},\ldots,B_{n}$ be algebraic groups.\footnote{That is, a group
schemes of finite type over a field.
\par
{}} A subdirect product of $B_{1}\times\cdots\times B_{n}$ is an algebraic
subgroup $A$ such that the projections $A\rightarrow B_{i}$ are all faithfully flat.

\begin{glemma}
\label{j20} Let $A$ be a subdirect product of $B_{1}\times B_{2}$, and let
$N_{1}$ and $N_{2}$ be the kernels of the projections $A\rightarrow B_{2}$ and
$A\rightarrow B_{1}$ regarded as (normal) algebraic subgroups of $B_{1}$
and~$B_{2}$. Let $\bar{A}$ be the image of $A$ in $B_{1}/N_{1}\times
B_{2}/N_{2}$. Then the projections $\bar{A}\rightarrow B_{1}/N_{1}$ and
$\bar{A}\rightarrow B_{2}/N_{2}$ are isomorphisms.
\end{glemma}

\begin{proof}
By symmetry, it suffices to show that the projection $\bar{A}\rightarrow
B_{1}/N_{1}$ is an isomorphism. It is a faithfully flat by hypothesis, and we
prove that it is injective by showing that the two homomorphisms
$A\rightarrow\bar{A}\rightarrow B_{1}/N_{1}$ and $A\rightarrow\bar{A}$ have
the same kernel. For a $k$-algebra $R$, the $R$-points of the kernels are%
\[
\left\{  (a_{1},a_{2})\in A(R)\mid a_{1}\in N_{1}(R)\right\}  \text{ and
}\{(a_{1},a_{2})\in A(R)\mid a_{1}\in N_{1}(R)\text{ and }a_{2}\in
N_{2}(R)\}.
\]
Let $(a_{1},a_{2})$ lie in the first kernel. To say that $a_{1}\in N_{1}(R)$
means that $(a_{1},1)\in\Ker(A(R)\rightarrow B_{2}(R))$. In particular,
$(a_{1},1)\in A(R)$. Thus $(1,a_{2})\in A(R)$, and so it lies in $N_{2}(R)$.
Hence $(a_{1},a_{2})\in N_{1}(R)\times N_{2}(R)$, and so it lies in the second kernel.
\end{proof}

The lemma says that the image $\bar{A}$ of $A$ in $B_{1}/N_{1}\times
B_{2}/N_{2}$ is the graph of an isomorphism $B_{1}/N_{1}\simeq B_{2}/N_{2}$.
If $B_{1}$ and $B_{2}$ are almost-simple, then either $A=B_{1}\times B_{2}$ or
it is the graph of an isogeny $B_{1}\rightarrow B_{2}$.

\begin{lemma}
\label{j21}Let $A$ be a normal algebraic subgroup of $B_{1}\times B_{2}$ such
that the projections $A\rightarrow B_{1}$ and $A\rightarrow B_{2}$ are
isomorphisms. Then $B_{1}$ and $B_{2}$ are commutative.
\end{lemma}

\begin{proof}
By assumption, $A$ is the graph of an isomorphism $\varphi\colon
B_{1}\rightarrow B_{2}$. Let $R$ be a $k$-algebra, and let $(b_{1}%
,\varphi(b_{1}))\in A(R)$. For all $b_{2}\in B_{2}(R)$, $(1,b_{2})$ normalizes
$(b_{1},\varphi(b_{1}))$, and so $b_{2}$ centralizes $\varphi(b_{1})$. This
shows that $B_{2}$ is commutative.
\end{proof}

\begin{lemma}
\label{j22}Let $A$ be a subdirect product of $B_{1}\times\cdots\times B_{n}$.
If $A$ is normal in $B_{1}\times\cdots\times B_{n}$ and the $B_{i}$ are
perfect, then $A=B_{1}\times\cdots\times B_{n}$.
\end{lemma}

\begin{proof}
We first take $n=2$. With the notation of Goursat's lemma, the algebraic
subgroup $\bar{A}$ of $B_{1}/N_{1}\times B_{2}/N_{2}$ satisfies the hypotheses
of the last lemma, and so $B_{1}/N_{1}$ and $B_{2}/N_{2}$ are commutative,
hence trivial. As $A$ contains $N_{1}$ and $N_{2}$, it equals $B_{1}\times
B_{2}$.

We prove the general case by induction on $n$. On applying the induction
hypothesis to the image of $A$ in $B_{1}\times\cdots\times B_{n-1}$, we find
that the projection map $A\rightarrow B_{1}\times\cdots\times B_{n-1}$ is
faithfully flat. Now $A$ is a subdirect product of $(B_{1}\times\cdots\times
B_{n-1})\times B_{n}$, and the case $n=2$ shows that $A=B_{1}\times
\cdots\times B_{n}.$
\end{proof}

\begin{lemma}
\label{j23}Let $A$ be an algebraic subgroup of $B_{1}\times\cdots\times B_{n}%
$. If each $B_{i}$ is perfect and each projection $A\rightarrow B_{i}\times
B_{j}$, $1\leq i<j\leq n$, is faithfully flat, then $A=B_{1}\times\cdots\times
B_{n}$.
\end{lemma}

\begin{proof}
The lemma is certainly true for $n=2$. Assume that $n>2$, and that the
statement is true for $n-1$. Then the projection $A\rightarrow B_{1}%
\times\cdots\times B_{n-1}$ is faithfully flat, and so the kernel $N$ of the
projection $A\rightarrow B_{n}$ is a normal algebraic subgroup of $B_{1}%
\times\cdots\times B_{n-1}$. By the last lemma, it equals $B_{1}\times
\cdots\times B_{n-1}$, and so there is an exact commutative diagram%
\[
\begin{tikzcd}
1\arrow{r}
&B_{1}\times\cdots\times B_{n-1}\arrow{r}\arrow[equals]{d}
&A\arrow{r}\arrow[hook]{d}
&B_{n}\arrow{r}\arrow[equals]{d}&1\\
1\arrow{r}
&B_{1}\times\cdots\times B_{n-1}\arrow{r}
&B_1\times\cdots\times B_n\arrow{r}
&B_{n}\arrow{r}&1,
\end{tikzcd}
\]
from which it follows that $A=B_{1}\times\cdots\times B_{n}$.
\end{proof}

The results in this section are well-known.

\subsection{Proof of the equality $P^{K}=L^{K}\cap S^{K}$.}

Let $K$ be a CM field, Galois over $\mathbb{Q}$, with a given $p$-adic prime.
Let $\Gamma=\Gal(K/\mathbb{Q})$, and let $D\subset\Gamma$ be the decomposition
group of the given prime. We write $P,S,L,T$ for $P^{K},S^{K},L^{K},T^{K}$.
Recall that we have diagrams%
\[
\begin{tikzcd}
T&S\arrow{l}&X^\ast (T)\arrow{d}{r}\arrow{r}{j}&X^\ast(S)\arrow{d}{s}\\
L\arrow{u}&P\arrow{u}\arrow{l}&X^\ast (L)\arrow{r}{i}&X^\ast(P)
\end{tikzcd}
\]
and have to show that the induced map $S/P\rightarrow T/L$ is injective or,
equivalently, that the map%
\[
j\colon\Ker(r)\rightarrow\Ker(s)
\]
is surjective.

Recall that $\iota$ denotes complex conjugation (on $K$ say), and so
$\langle\iota\rangle$ is the subgroup $\{1,\iota\}$ of $\Gamma$. For a finite
set $Y$ with an action of $\langle\iota\rangle$ and a positive integer $d$,
define%
\begin{align*}
\mathbb{Z}{}[Y]  &  =\text{free abelian group on }Y=\{\text{maps }f\colon
Y\rightarrow\mathbb{Z}\text{ or sums }\sum f(y)y\mathbb{\}}\text{,}\\
\mathbb{Z}{}[Y]^{d}  &  =\{f\in\mathbb{Z}{}[Y]\mid\exists c\in\mathbb{Z}\text{
such that }f+\iota f=d\cdot c\text{ (constant function)}\},\\
\mathbb{Z}{}[Y]_{0}  &  =\mathbb{Z}{}[Y]/\{f\mid f=\iota f\text{ and }\sum
f(y)=0\}.
\end{align*}
Then (see \S \S 1--5 of the \cite{milne1999lm}; also \cite{milne2001},
especially the diagram in A.8) the second of the above diagrams can be
identified with%
\[
\begin{tikzcd}
\mathbb{Z}[\mathcal{S}]_0\arrow{r}{j}\arrow{d}{r}
&\mathbb{Z}[\Gamma]^1\arrow{d}{s}\\
\mathbb{Z}[{\mathcal{P}}]_0\arrow{r}{i}
&\mathbb{Z}[\Gamma/D]^d.
\end{tikzcd}
\]
Here $\mathcal{S}{}$ is the set of CM-types on $K$, i.e., functions
$\varphi\colon\Gamma\rightarrow\{0,1\}$ such that $\varphi+\iota\varphi=1$,
and $\mathcal{P}$ is the set\footnote{Let $v$ be the given $p$-adic prime on
$K$, so that $d=[K_{v}\colon\mathbb{Q}{}_{p}]$. The map $\tau\mapsto\tau v$
defines a bijection of $\Gamma/D$ onto the set $X$ of $p$-adic primes of $K$.
Let $\pi\in W(p^{\infty})$, and let $s_{\pi}\colon X\rightarrow\mathbb{Q}$ be
the corresponding slope function (Milne 2001, A.6). For $\pi\in$
$W^{K}(p^{\infty})$, $d\cdot s_{\pi}$ takes values in $\mathbb{Z}{}$, and the
correspondence%
\[
d\cdot s_{\pi}\leftrightarrow\pi
\]
identifies $\mathcal{P}{}$ with the set of integral elements of weight $1$ (or
maybe $-1$, depending on conventions) in $W^{K}(p^{\infty})$, i.e., with the
Weil numbers in the sense of Tate 1968/69 modulo roots of $1$ corresponding to
abelian varieties over $\mathbb{F}{}$ whose endomorphism algebra is split by
$K$.} of functions $\pi\colon\Gamma/D\rightarrow\{0,1,\ldots,d\}$,
$d=(D\colon1)$, such that $\pi+\iota\pi=d$. The horizontal maps send a formal
sum to a sum of functions, e.g., $j$ sends the formal sum $\sum f(\varphi
)\varphi$ to the function $\tau\mapsto\sum f(\varphi)\varphi(\tau)\colon
\Gamma\rightarrow\mathbb{Z}{}$. The map $r $ sends $\varphi$ to $r(\varphi)$
where $r(\varphi)(\tau D)=\sum_{\sigma\in D}\varphi(\tau\sigma)$, and $s$
sends $f$ to $s(f)$ where $s(f)(\tau D)=\sum_{\sigma\in D}f(\tau\sigma)$.

\textbf{Case I: }$\iota\in D$. Choose a set $B$ of coset representatives for
$\langle\iota\rangle$ in $D$ and a set $A$ of coset representatives for $D$ in
$\Gamma$. Then, as a map of $\langle\iota\rangle$-sets, $\Gamma\rightarrow
\Gamma/D$ can be identified with the projection map%
\[
A\times B\times\{0,1\}\rightarrow A\times\{0,1\}\text{.}%
\]
Here $\iota$ acts only on $\{0,1\}$. Thus, the diagram becomes%
\[
\begin{CD} \mathbb{Z}[\mathcal{S}]_0 @>>>\mathbb{Z}[A\times B\times\{0,1\}]^1\\ @VVrV@VVsV\\ \mathbb{Z}[\mathcal{P}]_0@>>>\mathbb{Z}[A\times\{0,1\}]^d.\\ \end{CD}
\]
With this notation, a CM-type on $K$ is a function $\varphi\colon A\times
B\times\{0,1\}\rightarrow\{0,1\}$ such that
\[
\varphi(a,b,0)+\varphi(a,b,1)=1,\quad\text{all }a,b\text{,}%
\]
and an element of $\mathcal{P}$ is a function $\pi\colon A\times
\{0,1\}\rightarrow\{0,\ldots,d\}$ such that%
\[
\pi(a,0)+\pi(a,1)=d,\quad\text{all }a.
\]

For each $a,b$, let $\varphi_{a,b}$ be the CM-type such that
\[
\varphi_{a,b}(x,y,0)=1\iff(x,y)=(a,b),
\]
and let $\varphi^{\prime}$ be the CM-type such that
\[
\varphi(x,y,z)=z.
\]
Then $\varphi^{\prime}$ and the $\varphi_{a,b}$ generate $\mathbb{Z}{}[A\times
B\times\{0,1\}]^{1}$: if $f$ is a function on $A\times B\times\{0,1\}$ such
that $f+\iota f=c$ ($c\in\mathbb{Z}{}$), then%
\[
f=\sum_{a,b}f(a,b,0)\varphi_{a,b}+\Big(  c-\sum_{a,b}f(a,b,0)\Big)
\varphi^{\prime}\text{.}%
\]

Let $\pi_{a}=r(\varphi)_{a,b}$ and $\pi^{\prime}=r(\varphi^{\prime})$: thus
$\pi_{a}$ is the element of $\mathcal{P}$ such that $\pi_{a}(x,0)=1$ for $x=a$
and is zero otherwise, and $\pi^{\prime}(x,z)=d\cdot z$. Clearly, the $\pi
_{a}$ and $\pi^{\prime}$ are linearly independent.

We now prove this case of the theorem. Let $f\in\mathbb{Z}{}[A\times
B\times\{0,1\}]^{1}$, and write it%
\[
f=\sum_{a,b}n(a,b)\cdot\varphi_{a,b}+n\cdot\varphi^{\prime},
\]
so%
\[
s(f)=\sum_{a}\Big(  \sum_{b}n(a,b)\Big)  \pi_{a}+n\pi^{\prime}\text{.}%
\]
If $s(f)=0$, then, because of the linear independence of the $\pi_{a}$ and
$\pi^{\prime}$,
\begin{align*}
\sum_{b\in B}n(a,b)  &  =0,\quad\text{all }a\in A\text{, and}\\
n  &  =0\text{.}%
\end{align*}
The first equation implies that $\sum_{b}n(a,b)\varphi_{a,b}$ is in the kernel
of $r$, which completes the proof.

\textbf{Case II: }$\iota\notin D$. This case is so similar to the preceding
that it should be left as an exercise to the reader.

\bibliographystyle{cbe}
\bibliography{D:/Current/refs}

\def\Dbar{\leavevmode\lower.6ex\hbox to 0pt{\hskip-.23ex \accent"16\hss}D}
\begin{thebibliography}{}

\bibitem[\protect\astroncite{Ancona}{2021}]{ancona2021}
{\sc Ancona, G.} 2021.
\newblock Standard conjectures for abelian fourfolds.
\newblock {\em Invent. Math.} 223:149--212.

\bibitem[\protect\astroncite{Buskin}{2019}]{buskin2019}
{\sc Buskin, N.} 2019.
\newblock Every rational {H}odge isometry between two {$K3$} surfaces is
  algebraic.
\newblock {\em J. Reine Angew. Math.} 755:127--150.

\bibitem[\protect\astroncite{Chi}{1991}]{chi1991}
{\sc Chi, W.~C.} 1991.
\newblock On the {$l$}-adic representations attached to simple abelian
  varieties of type {${\rm IV}$}.
\newblock {\em Bull. Austral. Math. Soc.} 44:71--78.

\bibitem[\protect\astroncite{Clozel}{1999}]{clozel1999}
{\sc Clozel, L.} 1999.
\newblock Equivalence num\'erique et \'equivalence cohomologique pour les
  vari\'et\'es ab\'eliennes sur les corps finis.
\newblock {\em Ann. of Math. (2)} 150:151--163.

\bibitem[\protect\astroncite{Commelin}{2019}]{commelin2019}
{\sc Commelin, J.} 2019.
\newblock The {M}umford-{T}ate conjecture for products of abelian varieties.
\newblock {\em Algebr. Geom.} 6:650--677.

\bibitem[\protect\astroncite{Deligne}{1980}]{deligne1980}
{\sc Deligne, P.} 1980.
\newblock La conjecture de {W}eil. {II}.
\newblock {\em Inst. Hautes \'Etudes Sci. Publ. Math.} pp. 137--252.

\bibitem[\protect\astroncite{Deligne}{1982}]{deligne1982}
{\sc Deligne, P.} 1982.
\newblock Hodge cycles on abelian varieties (notes by {J}.{S}. {M}ilne), pp.
  9--100.
\newblock {\em In} Hodge cycles, motives, and {S}himura varieties, volume 900
  of {\em Lecture Notes in Mathematics}. Springer-Verlag, Berlin-New York,
  Berlin.

\bibitem[\protect\astroncite{Deligne}{1989}]{deligne1989}
{\sc Deligne, P.} 1989.
\newblock Le groupe fondamental de la droite projective moins trois points, pp.
  79--297.
\newblock {\em In} Galois groups over $\mathbb{Q}$ (Berkeley, CA, 1987),
  volume~16 of {\em Math. Sci. Res. Inst. Publ.} Springer, New York.

\bibitem[\protect\astroncite{Deligne and Milne}{1982}]{deligneM1982}
{\sc Deligne, P. and Milne, J.~S.} 1982.
\newblock Tannakian categories, pp. 101--228.
\newblock {\em In} Hodge cycles, motives, and {S}himura varieties, Lecture
  Notes in Mathematics 900. Springer-Verlag, Berlin.

\bibitem[\protect\astroncite{Faltings}{1983}]{faltings1983}
{\sc Faltings, G.} 1983.
\newblock Endlichkeitss\"atze f\"ur abelsche {V}ariet\"aten \"uber
  {Z}ahlk\"orpern.
\newblock {\em Invent. Math.} 73:349--366.
\newblock (Erratum Invent. Math. 75 (1984), p381).

\bibitem[\protect\astroncite{Hazama}{1989}]{hazama1989}
{\sc Hazama, F.} 1989.
\newblock Algebraic cycles on nonsimple abelian varieties.
\newblock {\em Duke Math. J.} 58:31--37.

\bibitem[\protect\astroncite{Imai}{1976}]{imai1976}
{\sc Imai, H.} 1976.
\newblock On the {H}odge groups of some abelian varieties.
\newblock {\em K\={o}dai Math. Sem. Rep.} 27:367--372.

\bibitem[\protect\astroncite{Ito et~al.}{2021}]{itoIT2021}
{\sc Ito, K., Ito, T., and Koshikawa, T.} 2021.
\newblock C{M} liftings of {$K3$} surfaces over finite fields and their
  applications to the {T}ate conjecture.
\newblock {\em Forum Math. Sigma} 9:Paper No. e29, 70.

\bibitem[\protect\astroncite{Jannsen}{1992}]{jannsen1992}
{\sc Jannsen, U.} 1992.
\newblock Motives, numerical equivalence, and semi-simplicity.
\newblock {\em Invent. Math.} 107:447--452.

\bibitem[\protect\astroncite{Kleiman}{1968}]{kleiman1968}
{\sc Kleiman, S.~L.} 1968.
\newblock Algebraic cycles and the {W}eil conjectures, pp. 359--386.
\newblock {\em In} Dix espos\'es sur la cohomologie des sch\'emas.
  North-Holland, Amsterdam.

\bibitem[\protect\astroncite{Lesin}{1994}]{lesin1994}
{\sc Lesin, A.~A.} 1994.
\newblock On the {M}umford-{T}ate conjecture for abelian varieties with
  reduction conditions.
\newblock ProQuest LLC, Ann Arbor, MI.
\newblock Thesis (Ph.D.)--California Institute of Technology.

\bibitem[\protect\astroncite{Lieberman}{1968}]{lieberman1968}
{\sc Lieberman, D.~I.} 1968.
\newblock Numerical and homological equivalence of algebraic cycles on {H}odge
  manifolds.
\newblock {\em Amer. J. Math.} 90:366--374.

\bibitem[\protect\astroncite{Markman}{2021}]{markman2021}
{\sc Markman, E.} 2021.
\newblock The monodromy of generalized {K}ummer varieties and algebraic cycles
  on their intermediate {J}acobians.
\newblock arXiv:1805.11574v3.

\bibitem[\protect\astroncite{Milne}{1999a}]{milne1999lc}
{\sc Milne, J.~S.} 1999a.
\newblock Lefschetz classes on abelian varieties.
\newblock {\em Duke Math. J.} 96:639--675.

\bibitem[\protect\astroncite{Milne}{1999b}]{milne1999lm}
{\sc Milne, J.~S.} 1999b.
\newblock Lefschetz motives and the {T}ate conjecture.
\newblock {\em Compositio Math.} 117:45--76.

\bibitem[\protect\astroncite{Milne}{2001}]{milne2001}
{\sc Milne, J.~S.} 2001.
\newblock The {T}ate conjecture for certain abelian varieties over finite
  fields.
\newblock {\em Acta Arith.} 100:135--166.

\bibitem[\protect\astroncite{Milne}{2002}]{milne2002p}
{\sc Milne, J.~S.} 2002.
\newblock Polarizations and {G}rothendieck's standard conjectures.
\newblock {\em Ann. of Math. (2)} 155:599--610.

\bibitem[\protect\astroncite{Milne}{2007}]{milne2007qtc}
{\sc Milne, J.~S.} 2007.
\newblock Quotients of {T}annakian categories.
\newblock {\em Theory Appl. Categ.} 18:No. 21, 654--664.

\bibitem[\protect\astroncite{Milne}{2013}]{milne2013m}
{\sc Milne, J.~S.} 2013.
\newblock Shimura varieties and moduli, pp. 467--548.
\newblock {\em In} Handbook of moduli. {V}ol. {II}, volume~25 of {\em Adv.
  Lect. Math. (ALM)}. Int. Press, Somerville, MA.

\bibitem[\protect\astroncite{Pohlmann}{1968}]{pohlmann1968}
{\sc Pohlmann, H.} 1968.
\newblock Algebraic cycles on abelian varieties of complex multiplication type.
\newblock {\em Ann. of Math. (2)} 88:161--180.

\bibitem[\protect\astroncite{Ribet}{1983}]{ribet1983}
{\sc Ribet, K.~A.} 1983.
\newblock Hodge classes on certain types of abelian varieties.
\newblock {\em Amer. J. Math.} 105:523--538.

\bibitem[\protect\astroncite{Spie{\ss}}{1999}]{spiess1999}
{\sc Spie{\ss}, M.} 1999.
\newblock Proof of the {T}ate conjecture for products of elliptic curves over
  finite fields.
\newblock {\em Math. Ann.} 314:285--290.

\bibitem[\protect\astroncite{Tate}{1966}]{tate1966e}
{\sc Tate, J.~T.} 1966.
\newblock Endomorphisms of abelian varieties over finite fields.
\newblock {\em Invent. Math.} 2:134--144.

\bibitem[\protect\astroncite{Tate}{1968}]{tate1968}
{\sc Tate, J.~T.} 1968.
\newblock Classes d'isog\'enie des vari\'et\'es ab\'eliennes sur un corps fini
  (d'apr\`es {T}. {H}onda).
\newblock S\'{e}minaire Bourbaki: Vol. 1968/69, Expose 352.

\bibitem[\protect\astroncite{Tate}{1994}]{tate1994}
{\sc Tate, J.~T.} 1994.
\newblock Conjectures on algebraic cycles in {$l$}-adic cohomology, pp. 71--83.
\newblock {\em In} Motives (Seattle, WA, 1991), volume~55 of {\em Proc. Sympos.
  Pure Math.} Amer. Math. Soc., Providence, RI.

\bibitem[\protect\astroncite{van Geemen}{1994}]{geemen1994}
{\sc van Geemen, B.} 1994.
\newblock An introduction to the {H}odge conjecture for abelian varieties, pp.
  233--252.
\newblock {\em In} Algebraic cycles and {H}odge theory ({T}orino, 1993), volume
  1594 of {\em Lecture Notes in Math.} Springer, Berlin.

\bibitem[\protect\astroncite{Vasiu}{2008}]{vasiu2008}
{\sc Vasiu, A.} 2008.
\newblock Some cases of the {M}umford-{T}ate conjecture and {S}himura
  varieties.
\newblock {\em Indiana Univ. Math. J.} 57:1--75.

\bibitem[\protect\astroncite{Weil}{1948}]{weil1948}
{\sc Weil, A.} 1948.
\newblock Vari{\'e}t{\'e}s ab{\'e}liennes et courbes alg{\'e}briques.
\newblock Actualit{\'e}s Sci. Ind., no. 1064 = Publ. Inst. Math. Univ.
  Strasbourg 8 (1946). Hermann {\&} Cie., Paris.

\bibitem[\protect\astroncite{Weil}{1977}]{weil1977}
{\sc Weil, A.} 1977.
\newblock Abelian varieties and the {H}odge ring.
\newblock Talk at a conference held at Harvard in honor of Lars Ahlfors (\OE
  uvres Scientifiques 1977c, 421--429).

\bibitem[\protect\astroncite{Zarhin}{1994}]{zarhin1994}
{\sc Zarhin, Y.~G.} 1994.
\newblock The {T}ate conjecture for nonsimple abelian varieties over finite
  fields, pp. 267--296.
\newblock {\em In} Algebra and number theory ({E}ssen, 1992). de Gruyter,
  Berlin.

\end{thebibliography}

\end{document}